\documentclass[a4paper,reqno, 12pt
]{amsart} 
\usepackage[english]{babel}
\usepackage[utf8]{inputenc}
\usepackage{microtype}
\usepackage{amssymb, graphicx, enumerate, enumitem, color}
\usepackage[usenames,dvipsnames,svgnames,table]{xcolor}
\usepackage[foot]{amsaddr}
\usepackage{thmtools}
\usepackage{verbatim}

\colorlet{mylinkcolor}{violet}
\colorlet{mycitecolor}{YellowOrange}
\colorlet{myurlcolor}{Aquamarine}
\usepackage[unicode=true]{hyperref}
\hypersetup{
    colorlinks,
    linkcolor={red!50!black},
    citecolor={blue!50!black},
    urlcolor={blue!80!black}
}

\usepackage[ruled,titlenumbered,linesnumbered,noend,norelsize]{algorithm2e}
\DontPrintSemicolon
\SetNlSty{footnotesize}{}{}
\IncMargin{0.5em}
\SetNlSkip{1em}
\SetKwIF{If}{ElseIf}{Else}{if}{then}{else if}{else}{endif}
\SetKwFor{Loop}{Loop}{}{EndLoop}
\SetKw{Return}{return}
\SetNlSkip{2em}
\SetKwProg{Proc}{Procedure}{}{}

\newenvironment{algorithm-hbox}{\hbadness=10000\begin{algorithm}}{\end{algorithm}}

\newtheorem{theorem}{Theorem}
\newtheorem{conjecture}[theorem]{Conjecture}
\newtheorem{corollary}[theorem]{Corollary}

\newtheorem{proposition}[theorem]{Proposition}

\newtheorem{lemma}[theorem]{Lemma}
\theoremstyle{remark}

\newcommand{\set}[1]{\{#1\}}

\DeclareMathOperator\Inc{Inc}

\DeclareMathOperator\Min{Min}
\DeclareMathOperator\Max{Max}

\DeclareMathOperator\wcol{wcol}

\DeclareMathOperator\se{se}

\DeclareMathOperator\WReach{WReach}

\let\le\leqslant
\let\ge\geqslant
\let\leq\leqslant
\let\geq\geqslant

\let\subset\subseteq
\let\subsetneq\varsubsetneq

\let\epsilon\varepsilon

\renewenvironment{enumerate}{\begin{enumorig}[label=\textup{(\roman*)}, noitemsep, topsep=2pt plus 2pt, labelindent=.2em, leftmargin=*, widest=iii]}{\end{enumorig}}

\renewenvironment{itemize}{\begin{itemorig}[label=\textbullet, noitemsep, topsep=2pt plus 2pt, labelindent=.5em, labelsep=.5em, leftmargin=*]}{\end{itemorig}}

\makeatletter
\let\old@setaddresses\@setaddresses
\def\@setaddresses{\bigskip\bgroup\parindent 0pt\let\scshape\relax\old@setaddresses\egroup}
\makeatother

\begin{document}
\title[POSETS WITH PLANAR COVER GRAPHS]{Dimension is Polynomial in Height\\
for Posets with Planar Cover Graphs}

\author[Kozik]{Jakub Kozik}
\email{jakub.kozik@uj.edu.pl}

\author[Micek]{Piotr Micek}
\address[J.\ Kozik, P.\ Micek]{Theoretical Computer Science Department\\
  Faculty of Mathematics and Computer Science, Jagiellonian University, 
  Krak\'ow, Poland}
\email{piotr.micek@uj.edu.pl}

\author[Trotter]{William T.\ Trotter}
\address[W.\ Trotter]{School of Mathematics\\
  Georgia Institute of Technology\\
  Atlanta, Georgia 30332\\
  U.S.A.}
\email{trotter@math.gatech.edu}

\thanks{P.\ Micek is partially supported by a Polish National Science 
  Center grant (SONATA BIS 5; UMO-2015/18/E/ST6/00299).}
\thanks{W.\ T.\ Trotter is partially supported by a Simons Collaboration grant.}

\newcommand{\cgA}{\mathcal{A}}
\newcommand{\cgB}{\mathcal{B}}
\newcommand{\cgC}{\mathcal{C}}
\newcommand{\cgD}{\mathcal{D}}
\newcommand{\cgF}{\mathcal{F}}
\newcommand{\cgG}{\mathcal{G}}
\newcommand{\cgH}{\mathcal{H}}
\newcommand{\cgU}{\mathcal{U}}
\newcommand{\cgM}{\mathcal{M}}
\newcommand{\cgP}{\mathcal{P}}
\newcommand{\cgR}{\mathcal{R}}
\newcommand{\cgW}{\mathcal{W}}
\newcommand{\mmD}{\mathbb{D}}
\newcommand{\Spec}{\operatorname{Spec}}
\newcommand{\MD}{\operatorname{MD}}
\newcommand{\MDE}{\operatorname{MDE}}

\newcommand{\Oh}{\mathcal{O}}

\date{\today}



\begin{abstract}
We show that height~$h$ posets that have planar cover graphs 
have dimension $\Oh(h^6)$.  Previously, 
the best upper bound was $2^{\Oh(h^3)}$.
Planarity plays a key role in our arguments, 
since there are posets such that 
(1) dimension is exponential in height and 
(2) the cover graph excludes $K_5$ as a minor.
\end{abstract}

\maketitle

\section{Introduction}\label{sec:introduction}

In this paper, we study finite partially ordered sets, \emph{posets} for
short, and we assume that readers are familiar with the basics of the
subject, including chains and antichains; minimal and maximal elements;
height and width; order diagrams (also called Hasse diagrams); and
linear extensions.  For readers who are new to combinatorics on posets,
several of the recent research papers cited in our bibliography
include extensive background information.

Following the traditions of the subject, \textit{elements} of a poset 
are also called \textit{points}. 
Recall that when $P$ is a poset, an element $x$ is
\textit{covered} by an element $y$ in $P$ when $x<y$ in $P$ and
there is no element $z$ of $P$ with $x<z<y$ in $P$.  We associate
with $P$ an ordinary graph $G$, called the \textit{cover graph} of $P$,
defined as follows.  The vertex set of $G$ is the ground set of $P$,
and distinct elements (now also called \emph{vertices}) $x$ and $y$ are adjacent in $G$ when
either $x$ is covered by $y$ in $P$ or $y$ is covered by $x$ in $P$.
We consider the edges of $G$ oriented by the order relation of $P$, i.e.\
an edge $xy$ is oriented from $x$ to $y$ when $x$ is covered by $y$ in $P$.

Dushnik and Miller~\cite{DM41} defined the \emph{dimension} of a 
poset $P$, denoted $\dim(P)$, as the least positive integer $d$ such
that there are $d$ linear orders $L_1,\dots,L_d$ 
on the ground set of $P$ such that $x\le y$ in $P$ if
and only if $x\le y$ in $L_i$ for each $i\in\set{1,\dots,d}$.
In general, there are many posets that have the
same cover graph, and among them, there may be posets which
have markedly different values of height, width and dimension.
Indeed, it is somewhat surprising that we are able to bound any 
combinatorial property of a finite poset in terms of graph 
theoretic properties of its cover graph.

However, Streib and Trotter~\cite{ST14} proved that dimension is
bounded in terms of height for posets that have a planar cover graph.
This stands in sharp contrast with a number of well-known
families of posets that have height $2$ but unbounded dimension 
(e.g.\ the standard examples discussed below).
The result from~\cite{ST14} prompted researchers to investigate in greater depth 
connections between dimension and graph theoretic properties of 
cover graphs.  Subsequently, it has been shown that
dimension  is bounded in terms of height for posets whose
cover graphs:
\begin{itemize}
\item Have bounded treewidth, bounded genus, or more generally
exclude an apex-graph as minor~\cite{JMMTWW};
\item Exclude a fixed graph as a (topological) minor~\cite{Walczak17, MW15};
\item Belong to a fixed class with bounded
expansion~\cite{JMW18}.
\end{itemize}
Moreover, the existence of bounds for dimension of posets with cover graphs 
in a fixed class can say something about the sparsity of the class.
Joret, Micek, Ossona de Mendez, and Wiechert~\cite{JMOdMW} proved that
a monotone class of graphs is nowhere dense if and only if 
for every $h\geq 1$ and every $\varepsilon>0$, posets of height $h$ 
with $n$ elements whose cover graphs are in the class have dimension 
$\Oh(n^{\varepsilon})$.

The best upper bound to date on dimension in terms of height for posets
that have planar cover graphs is $2^{\Oh(h^3)}$.   This result
can be extracted from~\cite{JMOdMW} via connections between dimension
for posets and \textit{weak-coloring numbers} of their cover
graphs.  We will give additional details on this work in the next
section.

Our main theorem improves this exponential bound to one which is
polynomial in~$h$.

\begin{theorem}\label{thm:main}
If $P$ is a poset of height $h$ and the cover graph of $P$
is planar, then $\dim(P)=\Oh(h^6)$.
\end{theorem}

Planarity plays a crucial role in the existence of a polynomial bound. 
In~\cite{JMW_Planar_Posets}, Joret, Micek and Wiechert show that 
for each even integer $h\geq 2$, there is a height $h$ poset $P$
with dimension at least $2^{h/2}$ such that
the cover graph of $P$ excludes $K_5$ as a minor.

To discuss lower bounds, we pause to give the following construction
which first appears in~\cite{DM41}.  For each $n\ge2$, let $S_n$ be 
the height~$2$ poset with $\{a_1,a_2,\dots,a_n\}$ the set of minimal elements,
$\{b_1,b_2,\dots,b_n\}$ the set of maximal elements, and
$a_i<b_j$ in $S_n$ if and only if $i\neq j$. Posets in the
family $\{S_n:n\ge2\}$ are now called \emph{standard examples},
as $\dim(S_n)=n$ for every $n\ge2$.

To date, the best lower bound for the maximum dimension of 
a height~$h$ poset with a planar cover graph is $2h-2$,
and this bound comes from the ``double wheel'' construction 
given in~\cite{JMW_Planar_Posets}, and illustrated here
in Figure~\ref{fig:double-wheel}. To avoid clutter, we do not
show arrowheads in our figures.  Instead, we indicate directions
using color and accompanying narrative.  In this figure, the
black edges are oriented in each individual wheel from outside
to inside.  The elements of $\{a_1,\dots,a_n\}$ are minimal
elements so the red edges are oriented ``left-to-right'' and the
blue edges are oriented ``right-to-left.'' 
\begin{figure}
\begin{center}
\includegraphics[scale=1]{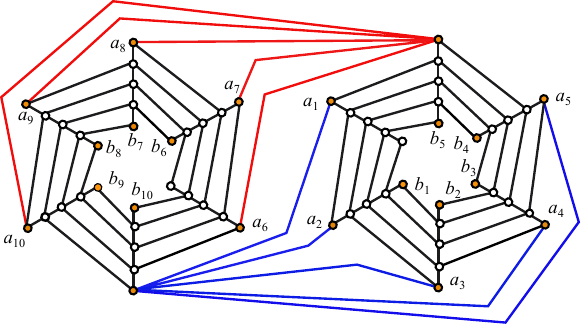}

\end{center}
\caption{We illustrate the double wheel construction when
$h=5$. Note that the elements 
$a_1,\ldots,a_{10}$ and $b_1,\ldots,b_{10}$ induce a standard 
example, so the dimension of the depicted poset is at least $10$. On the other
hand, the height of $P$ is~$5$.}
\label{fig:double-wheel}
\end{figure}

Requiring that the diagram of a poset $P$ is planar is a stronger restriction 
than requiring that the cover graph of $P$ is planar.
Accordingly among posets that have planar cover graphs, some but not
all also have planar order diagrams.  Among the class of posets with planar diagrams,
Joret, Micek and Wiechert~\cite{JMW_Planar_Posets} showed 
that $\dim(P)\leq 192h+96$ when $P$ has height~$h$.

The remainder of this paper is organized as follows.  In the next
section, we prove three reductions to simpler problems, and we
give essential background material.
The proof of Theorem~\ref{thm:main} is given in the following
two sections, and we close with brief comments on challenging
open problems that remain.

\section{Preliminary Reductions and Background Material}\label{sec:reductions}


We list below some elementary, and well known, properties of dimension.  

\begin{enumerate}
\item Dimension is monotonic, i.e., if $Q$ is an induced subposet of $P$,
then $\dim(Q)\le\dim(P)$.
\item The dual of a poset $P$ is the poset $P'$ on the same ground set
of $P$ with $x< y$ in $P'$ if and only $x>y$ in $P$.  Then $\dim(P)=\dim(P')$.
\end{enumerate}

For the balance of this preliminary section, we fix a poset $P$.
We let $\Min(P)$ and $\Max(P)$ denote, respectively, 
the set of minimal elements and the set of maximal elemets of $P$. 
When $x$ is in $P$, we let $U_P(x)$ consist of all elements $u$ 
such that $x<u$ in $P$.
Dually, $D_P(x)$ consist of all elements $u$ 
such that $x>u$ in $P$.

We write $x\parallel_P y$ (also $x\parallel y$ in $P$)
when $x$ and $y$ are incomparable. 
Also, we let $\Inc(P)$ 
denote the set of all ordered
pairs $(x,y)$ with $x\parallel_P y$.  We will assume $\Inc(P)\neq\emptyset$;
otherwise $P$ is a chain and $\dim(P)=1$.  
When $(x,y)\in\Inc(P)$ and $L$ is a linear extension of $P$, we say 
that $L$ \textit{reverses} $(x,y)$ when $y<x$ in $L$.
A set $I \subseteq \Inc(P)$ is \emph{reversible} if there is a 
linear extension $L$ of $P$ which reverses every pair in $I$. 
Vacuously, the empty set is reversible. 
We then define $\dim_P(I)$ as the least $d\ge1$ such that 
$I$ can be covered by $d$ reversible sets. 
It is easily seen that $\dim(P)$ is equal to $\dim_P(\Inc(P))$.

Given sets $A,B\subseteq P$, we let $\Inc_P(A,B)$ be the set of pairs 
$(a,b)\in\Inc(P)$ with $a\in A$ and $b\in B$.
We use the abbreviation $\dim_P(A,B)$ for $\dim_P(\Inc(A,B))$.
Again, $\dim_P(A,B)=1$ when $\Inc_P(A,B)=\emptyset$. 
When the meaning of the poset $P$ remains fixed in the discussion, 
we will drop the subscript in the notation $\dim_P(I)$ and 
$\dim_P(A,B)$.


A sequence $((x_1,y_1), \dots, (x_k,y_k))$ of pairs from $\Inc(P)$ 
with $k \geq 2$ is an \emph{alternating cycle of size $k$} 
if $x_i\leq_P y_{i+1}$ for all $i\in\set{1,\ldots,k}$, cyclically 
(so $x_k\le_P y_1$ is required).  
Observe that if $((x_1,y_1), \dots, (x_k,y_k))$ is an alternating 
cycle in $P$, then any subset $I\subseteq \Inc(P)$ containing
all the pairs on this cycle is not reversible.

An alternating cycle $((x_1,y_1),\dots,(x_k,y_k))$ is \emph{strict} if we have
$x_i\le_P y_{j}$ if and only if $j=i+1$ (cyclically).
Note that in this case, $\{x_1, \dots, x_k\}$ and
$\{y_1, \dots, y_k\}$ are $k$-element antichains.  Note also that
in alternating cycles, we allow that $x_i=y_{i+1}$ for some or even
all values of $i$.
Trotter and Moore~\cite{TM77} observed the following:
A subset $I\subseteq\Inc(P)$ is reversible if and only 
if $I$ contains no strict alternating cycle.

When $x<_P y$, a sequence $W=(u_0,u_1,\dots,u_t)$
is called a \textit{witnessing path} (\emph{from $x$ to $y$}) when $u_0=x$, $u_t=y$
and $u_i$ is covered by $u_{i+1}$ in $P$ for each $i\in\set{0,1,\dots,t-1}$.

The following elementary lemma allows us to concentrate our attention on 
incomparable pairs from $\Inc(\Min(P),\Max(P))$.
See for instance~\cite[Observation~3]{JMTWW} 
for a proof.

\begin{lemma}[Reduction to min-max]\label{lem:min-max-reduction}
For every poset $P$, 
there is a poset $Q$ containing $P$ as an induced 
subposet such that
\begin{enumerate}
 \item The height of $P$ is the same as the height of $Q$;
 \item The cover graph of $Q$ is obtained from the cover graph of $P$ by 
  adding some degree $1$ vertices; and
\end{enumerate}
\[
\dim(P)\leq \dim_Q(\Min(Q), \Max(Q)).
\]
\end{lemma}

\subsection{Constrained Subsets and Weak-Coloring Numbers}%
 \label{subsec:constrained-weakcol}

Let $P$ be a poset.
We say that a subset $I\subseteq \Inc(P)$ 
is \emph{singly constrained} in $P$ if
there is an element $x_0\in P$ 
such that $x_0 < b$ in $P$ for every $(a,b)\in I$.
To identify the element $x_0$, we will also say $I$ is singly constrained
\emph{by} $x_0$.  

The following lemma was used first in~\cite{ST14} for posets with
planar cover graphs and in a more complex form in \cite{JMTWW}.  
The underlying principle is the concept of \textit{unfolding}, which is 
an analogue of breadth first search for posets.

\begin{lemma}[Reduction to singly constrained]
\label{lem:unfolding}
For every non-empty poset $P$, 
there exists a poset $Q$ such that
\begin{enumerate}
  \item The height of $Q$ is at most the height of $P$.
  \item The cover graph of $Q$ is a minor of the cover graph of $P$.
  \item There is a minimal element $x_0$ in $Q$ such that 
  $x_0\leq q$ in $Q$ for all $q\in\Max(Q)$, and
\[
\dim_P(\Min(P),\Max(P)) \leq 2 \dim_Q(\Min(Q),\Max(Q)).
\]
\end{enumerate}
In particular, the set $\Inc(\Min(Q),\Max(Q))$ is 
singly constrained by $x_0$ in $Q$.
\end{lemma}

We say that a subset $I$ of $\Inc(P)$ 
is \emph{doubly constrained} in $P$ 
when there is a pair of elements $(x_0,y_0)$ in $P$ such
that 
\begin{enumerate}
\item $x_0< y_0$ in $P$,
\item $x_0 < b$ in $P$ for every $(a,b)\in I$, and
\item $a < y_0$ in $P$ for every $(a,b)\in I$.
\end{enumerate}
As before, we will also say that $I$ is 
doubly constrained \emph{by} $(x_0,y_0)$. 

We would very much like to reduce to the case where we are bounding $\dim(I)$
when $I\subset \Inc(P)$ is doubly constrained.  Unfortunately, 
Lemma~\ref{lem:unfolding} will not be of assistance.  Instead, we will use a different 
reduction, one that will cost us an $\Oh(h^3)$-factor in the final bound.

The \emph{length} of a path in a graph is the number of its edges.
For two vertices $u$ and $v$ in a graph $G$, an $u$--$v$ \emph{path} is a 
path in $G$ with ends in $u$ and $v$. 
Let $G$ be a graph, and let $\sigma$ be an ordering of the vertices of $G$. 
For $r\in\set{0,1,2,\ldots}\cup\set{\infty}$ and two vertices $u$ and $v$ of $G$, we say that
$u$ is \emph{weakly $r$-reachable} from $v$ in $\sigma$, if there exists
an $u$--$v$ path of length at most $r$ 
such that for every vertex $w$ on the path, $u\leq_{\sigma} w$.
The set of vertices that are weakly $r$-reachable from a vertex $v$ in $\sigma$ 
is denoted by $\WReach_r[G, \sigma, v]$. 
The {\em weak $r$-coloring number} $\wcol_r(G)$ of $G$ is defined as
\[
 \wcol_r(G):=\min_{\sigma} \max_{v\in V(G)} \left|\WReach_r[G, \sigma, v]\right|.
\]
where $\sigma$ ranges over the set of all vertex orderings of $G$.
We call $\wcol_r(G)$ the $r$-\emph{th weak coloring number} of $G$.

Weak coloring numbers were originally introduced by Kierstead and 
Yang~\cite{KY03} as a generalization of the degeneracy of a graph (also 
known as the coloring number). Since then, they have been applied
in several novel situations (see Zhu~\cite{Zhu09} and Van den Heuvel 
et al.~\cite{vdHKQ19}, for examples).
We also have good bounds on weak coloring numbers.
For planar graphs, van den Heuvel et al.~\cite{VdHOdMQRS17} have shown that
the $r$-th weak coloring number is at most $\binom{r+2}{2}\cdot(2r+1)=\Oh(r^3)$.
See also a recent paper~\cite{JM22} with a lower bound in $\Omega(r^2\log r)$.

Here is a lemma on weak coloring numbers from~\cite{JMOdMW} that will
play an important role in the reduction to the doubly constrained case.

\begin{lemma}\label{lem:q-support}
 Let $P$ be a height $h$ poset, let $G$ be the cover graph of
 $P$, and let $c:= \wcol_{4h-4}(G)$.  
 If $I\subseteq\Inc(P)$, 
 then there is an element $z_0\in P$ such that the set 
 $J= \{(a,b)\in I: a < z_0\ \text{in $P$} \}$ satisfies  
\[ 
  \dim(J)\geq \frac{\dim(I)}{c} - 2.
\] 
\end{lemma}

We then have the following immediate corollary.

\begin{corollary}\label{cor:doubly-constrained}
Let $P$ be a poset with a planar cover graph, 
and let $x_0$ be an element of $P$ such that $x_0 < b$ in $P$ for every 
$b\in\Max(P)$. 
Let $I$ be a subset of $\Inc(P)$ that is 
singly constrained by $x_0$. 
Then there is a set $J\subset I$ and 
an element $y_0$ of $P$ such that 
$J$ is doubly constrained by $(x_0,y_0)$ in $P$ and
\[
 \dim(I)=\Oh(h^3)\cdot\dim(J).
\]
\end{corollary}

\begin{proof}
Let $G$ be the cover graph of $P$.
Apply Lemma~\ref{lem:q-support} with 
$c= \wcol_{4h-4}(G) = \Oh(h^3)$ to obtain the element $z_0$ and the set $J\subset I$ such that 
$J=\set{(a,b)\in I\mid a<z_0\ \text{in $P$}}$ and $\dim(J)\geq\dim(I)/c-2$.
Let $y_0$ be any maximal element with $z_0 \leq y_0$ in $P$.
Since $y_0\in\Max(P)$ we have $x_0 < y_0$ in $P$.
Evidently, $J$ is doubly constrained by the pair $(x_0,y_0)$.  
The inequality from Lemma~\ref{lem:q-support} 
becomes $\dim(I)\le c\cdot(2+\dim(J))$, and with this
observation, the proof of the corollary is complete.
\end{proof}

\subsection{A Reduction to the Doubly Exposed Case}
Let $P$ be a poset and let $I\subseteq\Inc(P)$ that is doubly constrained by $(x_0,y_0)$ in $P$. 
For convenience, 
we will say that 
a sequence $(P,x_0,y_0,I)$ is \emph{doubly constrained}.

Let $(P,x_0,y_0,I)$ be doubly constrained. 
Furthermore, suppose that the cover graph $G$ of $P$ is planar. 
Fix a plane drawing $\mathcal{D}$ of $G$ with $x_0$ on the 
exterior face.
Also, we append an imaginary edge $e_{-\infty}$ 
attached to $x_0$ in the exterior face.

Let $z$ be an element of $P$, and 
let $e_0$ be an edge of $G$ incident to $z$.
With $(z,e_0)$ fixed, we consider all the edges incident to $z$  ordered by the 
clockwise traversal around $z$ starting at $e_0$---this constitutes the 
\emph{left-to-right} $(z,e_0)$-\emph{ordering} of edges around $z$. Thus, $e$ is 
\emph{left of} $e'$ in this ordering if the clockwise traversal around $z$
starting at $e_0$ visits 
$e$ before $e'$, see Figure~\ref{fig:x0-left-right}.

Let $b$ and $b'$ be distinct elements of $P$. 
Also, let $W$ and $W'$ be paths in $G$ from $x_0$ to $b$ and 
$x_0$ to $b'$, respectively. 
We say that $W$ and $W'$ are $x_0$-\emph{consistent} if 
there is an element $z\not\in\set{b,b'}$ common to $W$ and $W'$ such that 
(1) $x_0Wz=x_0W'z$; and 
(2) $z$ is the only vertex common to $zWb$ and $zW'b'$. 
If $z\neq x_0$, let $e_0$ be the last edge of $x_0Wz$. 
If $z=x_0$, let $e_0$ be the imaginary edge $e_{-\infty}$.
Let $e$ and $e'$ be the first edge of $zWb$ and $zW'b'$, respectively. 
Now, we say that $W$ is $x_0$-\emph{left} ($x_0$-\emph{right}) of $W'$ if $e$ is left (right) of $e'$ in the $(z,e_0)$-ordering. 
Note that either $W$ is $x_0$-left of $W'$ or $W$ is $x_0$-right of $W'$.
See Figure~\ref{fig:x0-left-right}.

\begin{figure}
\begin{center}
\includegraphics[scale=1]{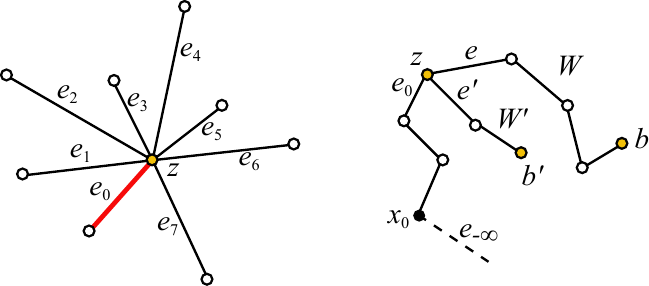}

\end{center}
\caption{Left: The edges incident to $z$ are enumerated with respect to 
the $(z,e_0)$-ordering.
Right: The paths $W$ and $W'$ are $x_0$-consistent, and 
$W$ is $x_0$-left of $W'$.
\label{fig:x0-left-right}}
\end{figure}

For the remainder of the paper, we say that 
a sequence $(P,G,x_0,y_0,I,\mathcal{D})$ is \emph{doubly exposed} if 
\begin{enumerate}
  \item $P$ is a poset and $G$ is the cover graph of $P$; 
  \item $x_0$ and $y_0$ are elements in $P$ with $x_0<y_0$; 
  \item $I$ is doubly constrained by $(x_0,y_0)$ in $P$;
  \item $G$ is planar and $\mathcal{D}$ is a plane drawing of $G$ with 
  $x_0$ and $y_0$ on the exterior face.
\end{enumerate}


\begin{lemma}\label{lem:double-exposed}
Let $(P,x_0,y_0,I)$ be doubly constrained, 
and assume that $P$ has a planar cover graph and height at most $h$.
Then there exists a sequence $(Q,G,u_0,v_0,J,\mathcal{D})$ that is doubly exposed and
\[
 \dim_P(I) \le 2(h-1)\dim_Q(J) +1.
\]
\end{lemma}

\begin{proof}
Fix a plane drawing of the
cover graph of $P$ with $x_0$ on the exterior face.
Fix also a witnessing path $W^{*}$ from $x_0$ to $y_0$ and refer to this chain as the \emph{spine}.
Label the points on the spine as $\{u_0,u_1,\dots,u_t\}$ with
$x_0=u_0$, $y_0=u_t$ and $u_i$ covered by $u_{i+1}$ in $P$ for
each $i\in\set{0,\dots,t-1}$.  Note that $t\le h-1$.

Recall that $I\subseteq \Inc(D_P(y_0),U_P(x_0))$. 
Note that the set $S=\set{(a,b)\in I\mid b\in W^*}$ is reversible (as you cannot build a strict alternating cycle with pairs within the set). 
Let $J=I-S$.
Thus, $\dim(I)\leq \dim(J)+1$. 

For each $b\in U_P(x_0)$, let $\tau(b)$ be the largest integer $i$
so that $u_i<_P b$.  Note that $0\le\tau(b)\leq t-1$.
Let $W_b$ be a witnessing path from $x_0$ to $b$ such that $W_b$
shares the initial segment $(u_0,u_1,\dots,u_{\tau(b)})$ with the
spine.
Note that $W^{*}$ and $W_b$ are $x_0$-consistent for each $b\in U_P(x_0)$ 
as long as $b\not\in W^*$.

We partition $U_P(x_0)-W^*$ into $B_{\textrm{left}}$ and $B_{\textrm{right}}$ in such a way that
$b$ is assigned to the set $B_{\textrm{left}}$ 
if $W_b$ is $x_0$-left of $W^{*}$. 
Dually,
we assign $b$ to $B_{\textrm{right}}$ if $W_b$ is $x_0$-right of $W^{*}$.

For each $a\in D_P(y_0)$, let
$\tau(a)$ be the least integer $i$ so that $a< u_i$ in $P$.  Now we have
$1\le \tau(a)\le t$. 
We partition the set $D_P(y_0)$ into $A_1\cup A_2\cup \cdots\cup A_t$ by assigning
$a$ to $A_i$ when $\tau(a)=i$.
Clearly,
\[
\dim(I) \leq 1+\dim(J) \leq 1+\sum_{s\in\set{1,\ldots,t}}\ 
\sum_{\text{dir}\in\set{\text{left}, \text{right}}} \dim(A_s,B_{\text{dir}}).
\]
It follows that there is some $s\in\set{1,\ldots,t}$ and 
$\text{dir}\in\set{\text{left}, \text{right}}$ so that
\[
\dim(I) \leq 1 + 2(h-1)\dim(A_s,B_{\textrm{dir}}).
\]

We assume that $\text{dir}=\text{right}$. 
From the details
of the argument, it will be clear that the proof is symmetric
in the other case.

We say that an edge $e=u_iv$ in the cover graph of $P$ is \emph{bad} 
if $0\le i<s$, $v$ is not on the spine, and 
$e$ is left of $u_iu_{i+1}$ in the $(u_i,u_{i-1}u_i)$-ordering.
(Note that $e=uv_i$ in the cover graph means $u_i < v$ in $P$ or $v<u_i$ in $P$.) 
We then define a poset $Q$ having 
the same ground set as $P$ with $x \leq y$ in $Q$ if and only if 
there is a witnessing path in $P$ from $x$ to $y$ avoiding bad edges.  

We claim that for every $a \in A_s$ and every $b\in B_{\text{right}}$, we have
$a \leq b$ in $Q$ if and only if $a \leq b$ in $P$. 
The forward implication is obvious.
To see the backward one,
let $a\in A_s$ and $b\in B_{\text{right}}$ with $a< b$ in $P$.
Then let $W$ be a witnessing path from $a$ to $b$ in $P$.
This path cannot use a bad edge as this would make $a < u_i$ in $P$ 
for some $i\in\set{1,\ldots,s-1}$ contradicting $a\in A_s$. 
Therefore, the claim holds and also
$\dim(A_s,B_{\text{right}})$ in $Q$ is the same as $\dim(A_s,B_{\text{right}})$ in $P$.

Note that the diagram and the cover graph of 
$Q$ are obtained simply by removing the bad edges from the diagram and
cover graph, respectively, of $P$.  It follows that the cover graph 
of $Q$ is planar.  Furthermore, $x_0$ and $u_s$ are on the same face,
and the set $\Inc(A_s,B_{\text{right}})$ is doubly exposed by the pair $(x_0,u_s)$. 
With this observation, the proof of the lemma is complete.
\end{proof}

Summarizing, we can combine Lemma~\ref{lem:unfolding}, 
Corollary~\ref{cor:doubly-constrained}, and Lemma~\ref{lem:double-exposed} 
to obtain:

\begin{corollary}\label{cor:doubly-exposed}
Let $P$ be a height $h$ poset with a planar cover graph.
Then there is a sequence $(Q,G,x_0,y_0,I,\mathcal{D})$ that is doubly exposed 
such that
$Q$ has height at most $h$ and
\[
 \dim(P)=\Oh(h^4)\cdot\dim_Q(I).
\]
\end{corollary}

The reader may note that the argument for the reduction 
actually proves that we may assume that the pairs in $I$ 
are min-max pairs. 
In order that our results can be applied in a more general setting, 
we elect to proceed with only the assumption that $I\subseteq \Inc(Q)$. 
is doubly exposed by $(x_0,y_0)$. 

We are now ready to begin the proof of our main theorem.

\section{Large Standard Examples in the Doubly Exposed Case}%
\label{sec:dim-and-breadth}

We pause here to make the following important comment:\quad
Height plays \emph{no} role in the arguments given in this
section.  

Throughout this section, 
we discuss a sequence $(P,G,x_0,y_0,I,\mathcal{D})$ that is 
doubly exposed.

We add to the drawing $\mathcal{D}$ an imaginary edge $e_{-\infty}$ in the exterior face to $x_0$, and we add
an imaginary edge $e_{+\infty}$ in the exterior face to $y_0$.

Let $B$ be an antichain in $P$ with $B\subseteq U_P(x_0)$.
Let $W_b$ be a witnessing path from $x_0$ to $b$, 
for each $b\in B$. 
We say that the family $\set{W_b\mid b\in B}$ is $x_0$-\emph{consistent} 
if $W_b$ and $W_{b'}$ are $x_0$-consistent whenever $b$ and $b'$ are distinct 
elements of $B$.
In this case, the edges of the paths in this family form a tree. 
We call this tree a \emph{witnessing tree} for $B$.
It is natural to use a single symbol, such as $T$, to denote this tree. 
Now for each $b\in B$, the path $W_b$ becomes $x_0Tb$. 
When $T$ is a witnessing tree for $B$, we write $b<_T b'$ 
when $x_0Tb$ is $x_0$-left of $x_0Tb'$. 
Note that $<_T$ is a linear order on elements of $B$.

We observe that if 
$B\subsetneq B'$ are both antichains in $U_P(x_0)$, 
and $T$ is a witnessing tree for $B$, then 
there is a witnessing tree $T'$ for $B'$ such that 
$x_0T'b=x_0Tb$ for all $b\in B$. 
In this case, $b<_T b'$ if and only if $b<_{T'}b'$ for all $b,b'\in B$.

Let $b$ and $b'$ be incomparable elements of $U_P(x_0)$.
We say that \emph{$b$ is $x_0$-left of $b'$} if 
$W$ is $x_0$-left of $W'$ whenever 
$W$ and $W'$ are $x_0$-consistent witnessing paths from $x_0$ to $b$ and $b'$, 
respectively. 
When $b$ is $x_0$-left of $b'$, we also say \emph{$b'$ is $x_0$-right of $b$}.

\begin{proposition}\label{pro:left-transitive}
Let $b$, $b'$, $b''$ be elements of $U_P(x_0)$. 
If $(b,b'),(b',b'')\in\Inc(P)$,  
$b$ is $x_0$-left of $b'$ and $b'$ is $x_0$-left of $b''$, then 
$(b,b'')\in\Inc(P)$ and $b$ is $x_0$-left of $b''$.
\end{proposition}

\begin{proof}
We first show that $(b,b'')\in\Inc(P)$. 
Let $W'$ be an arbitrary witnessing path from $x_0$ to $b'$. 
Then, let $W$ and $W''$ be witnessing paths from 
$x_0$ to $b$ and $b''$, respectively, such that 
$\set{W,W'}$ is $x_0$-consistent and    
$\set{W',W''}$ is $x_0$-consistent.
Let $w$ be the least element of $W$ which is not on $W'$. 
Also, let $w''$ be the least element of $W''$ which is not on $W'$. 
We claim that $wWb$ and $w''W''b''$ are disjoint.
To the contrary suppose that $z$ is a common element.
Now, $x_0WzW''b''$ and $W'$ are $x_0$-consistent and 
$x_0WzW''b''$ is $x_0$-left of $W'$,
contradicting the assumption that $b'$ is $x_0$-left of $b''$.

Now, suppose $b<b''$ in $P$ and let $U$ be a witnessing path 
from $b$ to $b''$. 
Then $x_0WbUb''$ and $W'$ are $x_0$-consistent, and 
$x_0WbUb''$ is $x_0$-left of $W'$, contradicting that 
$b'$ is $x_0$-left of $b''$. 
A symmetric shows that $b''\not< b'$ in $P$.
Thus, $(b,b'')\in\Inc(P)$.

Now we will argue that $b$ is $x_0$-left of $b''$. 
To the contrary, let $W$ and $W''$ be witnessing paths from $x_0$ 
to $b$ and $b''$, respectively, such that $W$ and $W''$ are 
$x_0$-consistent, and $W''$ is $x_0$-left of $W$. 
Choose a witnessing $W'$ from $x_0$ to $b'$ such that 
the union of $W$, $W'$, $W''$ form a witnessing tree. 
Then $W'$ is $x_0$-left of $W''$ and $x_0$-right of $W$. 
Clearly, this is impossible.
\end{proof}

In the doubly exposed setting $\mathbb{F}=(G,P,x_0,y_0,I,\cgD)$, 
we have two distinguished elements on the exterior face, namely $x_0$ and $y_0$. 
We proceed with a compact description of necessary definitions and notations involving $y_0$ 
that are dual to those introduced for $x_0$.

Let $a$ and $a'$ be distinct elements of $P$. 
Also, let $W$ and $W'$ be paths in $G$ from $a$ to $y_0$ and 
$a'$ to $y_0$, respectively. 
We say that $W$ and $W'$ are $y_0$-\emph{consistent} if 
there is an element $z\not\in\set{a,a'}$ common to $W$ and $W'$ such that 
(1) $y_0Wz=y_0W'z$ and 
(2) $z$ is the only vertex common to $zWa$ and $zW'a'$. 
If $z\neq y_0$, let $e_0$ be the last edge of $y_0Wz$. 
If $z=y_0$, let $e_0$ be the imaginary edge $e_{+\infty}$. 
Let $e$ and $e'$ be the first edge of $zWa$ and $zW'a'$, respectively. 
Now, we say that $W$ is $y_0$-\emph{left} ($y_0$-\emph{right}) of $W'$ if $e$ is left (right) of $e'$ in the $(z,e_0)$-ordering. 
Note that either $W$ is $y_0$-left of $W'$ or $W$ is $y_0$-right of $W'$.

Let $A$ be an antichain in $P$ with $A\subseteq D_P(y_0)$. 
Let $W_a$ be a witnessing path from $a$ to $y_0$, 
for each $a\in A$. 
We say that the family $\set{W_a\mid a\in A}$ is $y_0$-\emph{consistent} 
if $W_a$ and $W_{a'}$ are $y_0$-consistent whenever $a$ and $a'$ are distinct elements of $A$.
In this case, the edges of the paths in this family form a tree. 
Let $S$ be the obtained tree, 
we call this tree a \emph{witnessing tree} for $A$.
Now for each $a\in A$, the path $W_a$ becomes $aSy_0$. 
When $S$ is a witnessing tree for $A$, we write $a<_S a'$ 
when $aSy_0$ is $y_0$-left of $a'Sy_0$. 

Let $a$ and $a'$ be incomparable elements of $D_P(y_0)$.
We say that \emph{$a$ is $y_0$-left of $a'$} if 
$W$ is $y_0$-left of $W'$ whenever 
$W$ and $W'$ are $y_0$-consistent witnessing paths from $a$ to $y_0$ 
and from $a'$ to $y_0$, 
respectively. 
When $a$ is $y_0$-left of $a'$, 
we also say \emph{$a'$ is $y_0$-right of $a$}.
We state for emphasis a dual statement to Proposition~\ref{pro:left-transitive}. 

\begin{proposition}\label{pro:y0-left-transitive}
Let $a$, $a'$, $a''$ be elements of $D_P(y_0)$. 
If $(a,a'),(a',a'')\in\Inc(P)$,  
$a$ is $y_0$-left of $a'$ and $a'$ is $y_0$-left of $a''$, then 
$(a,a'')\in\Inc(P)$ and $a$ is $y_0$-left of $a''$.
\end{proposition}

When $\cgC$ is a simple closed curve in the plane, it splits
the points of the plane not on $\cgC$ into those that are in the
interior of the region bounded by $\cgC$ and those in the
exterior of this region.  In the discussion to follow, we will abuse 
terminology slightly and say that a point not on $\cgC$ is either
in the interior of $\cgC$ or it is in the exterior of $\cgC$,
dropping the reference to the region bounded by $\cgC$.
We then fix a simple closed curve $\cgC$ such that 
$x_0$ and $y_0$ are on $\cgC$ while all other vertices and 
edges of $G$ are in the interior of $\cgC$.


Let $N$ be a path from $x_0$ to $y_0$ in $G$. 
Then the clockwise portion of $\cgC$ beginning at $x_0$ and ending at $y_0$ 
together with $N$ traversed backwards is a simple closed curve. 
Elements of $G$ that are in the interior of this curve are said to be 
\emph{left of $N$}. 
Analogously, elements of $G$ that are not in $N$ and not left of $N$ are said to be 
\emph{right of $N$}.
These conventions are illustrated in Figure~\ref{fig:separating-path}.

The following self-evident proposition is stated for emphasis. 

\begin{proposition}\label{pro:self-evident}
Let $N$ be a path from $x_0$ to $y_0$ in $G$ and 
let $u$ be a vertex on $N$. 
If $u=x_0$, set $e_0$ to be the imaginary edge $e_{-\infty}$. 
Otherwise, set $e_0$ to be the last edge of $x_0Nu$. 
If $u=y_0$, set $e_1$ to be the imaginary edge $e_{+\infty}$. 
Otherwise, set $e_1$ to be the first edge of $uNy_0$. 
Let $N'$ be a non-trivial path starting at $u$ and 
let $e$ be the first edge of $N'$.
\begin{enumerate}
\item If $e$ left (right) of $e_1$ in the $(u,e_0)$-ordering, then 
$e$ is left (right) of $N$.
\item All edges and vertices of $N'$, except $u$, 
are on the same side of $N$ unless there is a vertex $u'\neq u$ 
that is common to $N$ and $N'$. 
\end{enumerate}
\end{proposition}

Let $N$ be a path from $x_0$ to $y_0$ in $G$. 
We call $N$ a \emph{separating path} if there exist (not necessarily distinct) elements 
$u$, $v$ of $N$ such that 
(1) $v\leq u$ in $P$ and $v$ does not occur before $u$ in traversing $N$ from $x_0$ to $y_0$;
(2) $x_0Nu$ and $vNy_0$ are witnessing paths; and 
(3) $uNv$ is a witnessing path from $v$ to $u$ traversed backwards.
To identify the elements $u$ and $v$ in this definition, 
we will write $N$ as $N(u,v)$.
When $N=N(u,v)$ is a separating path, 
we refer to $x_0Nu$ as the \emph{blue part} of $N$. 
Analogously, $vNy_0$ is the \emph{red part} of $N$, 
while $uNv$ is the \emph{black part} of $N$.
Note that if $z$ is on the path $N$, and $z$ is either on the red
part or the black part, then $v\le_P z$.  Symmetrically, if
$z$ is on the blue part or the black part, then $z\le_P u$.

A separating path $N(u,v)$ is \emph{associated with a comparability $a\leq b$} in $P$ 
when 
$a\leq v$, and $u\leq b$ in $P$.
The next proposition states that for all $a$, $b$ with $a\leq b$ in $P$, 
and all $W$, $W'$ witnessing paths from $x_0$ to $b$ and $b'$, respectively, 
we can find a separating path associated with $a\leq b$ that aligns well with $W$ and $W'$. 
See Figure~\ref{fig:separating-path}.

\begin{proposition}\label{pro:sep-path}
Let $(a,b)\in D_P(y_0)\times U_P(x_0)$ with $a\leq b$ in $P$. 
Let $W$ be a witnessing path from $a$ to $y_0$, and 
let $W'$ be a witnessing path from $x_0$ to $b$. 
Then there exists a separating path $N=N(u,v)$ associated with $a\leq b$ in $P$ 
such that 
(1) $u\in W'$ and the blue part of $N$ is $x_0W'u$; and 
(2) $v\in W$ and the red part of $N$ is $vWy_0$.
\end{proposition}
\begin{proof}
Suppose first that $W$ and $W'$ intersect 
and let $u$ be a common point. 
Then $N=x_0W'uWy_0$ satisfies the statement.
Otherwise, let $u$ be the least point in $P$ on $W$ such that 
$a<u$ in $P$. 
Then take $v$ to be the greatest point in $P$ on $W'$ such that 
$v<u$ in $P$.
Set $N_2$ to be an arbitrary witnessing path from $v$ to $u$ traversed backwards. 
Then $x_0W'uN_2vWy_0$ is the desired separating path.
\end{proof}

\begin{figure}
\begin{center}
\includegraphics[scale=.8]{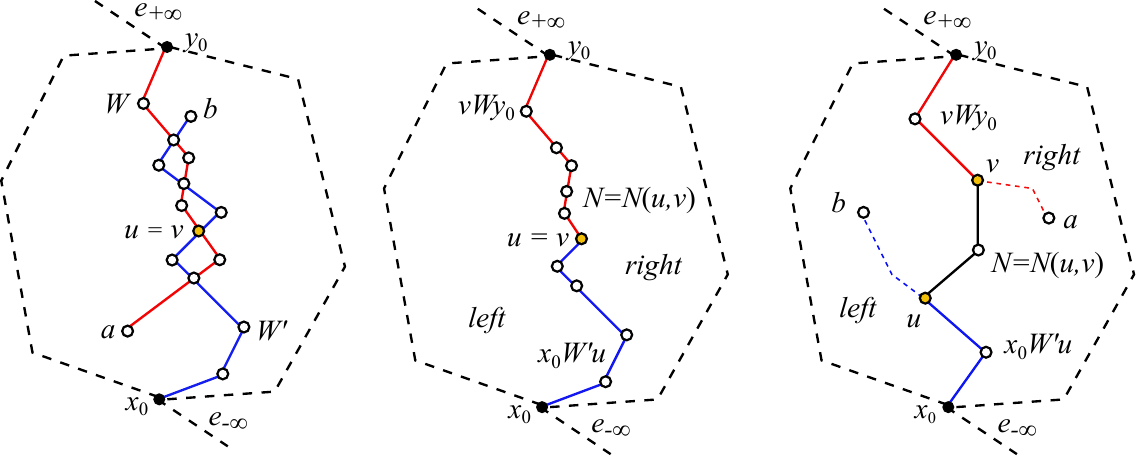}
\end{center}
\caption{
Left and Middle: A separating path $N$ for $a\leq b$ in $P$ 
such that the two fixed witnessing paths $W$ and $W'$ are intersecting.
Right: A separating path $N$ for $a\leq b$ in $P$ with 
$W$ and $W'$ disjoint.}
\label{fig:separating-path}
\end{figure}

The following elementary proposition has four symmetric statements: 
two for $D_P(y_0)$ and two for $U_P(x_0)$.

\begin{proposition}\label{pro:everything}
Suppose $a_1,a_2\in D_P(y_0)$, $b\in U_P(x_0)$, and $a_2\leq b$ in $P$.
For $i\in\set{1,2}$, let $W_i$ be a witnessing path from $a_i$ to $y_0$, and 
suppose that $W_1$ and $W_2$ are $y_0$-consistent. 
Suppose further that $W_1$ is $y_0$-left of $W_2$.
Let $N=N(u,v)$ be a separating path associated with $a_2\leq b$ in $P$ 
with $v$ on $W_2$ and $vNy_0=vW_2y_0$.
\begin{enumerate}
\item If $a_1\parallel b$ in $P$, then $a_1$ is right of $N$. 
\item If $a_1$ is not right of $N$, then $a_1\leq b$ in $P$, and $W_1$ contains a point from 
$x_0Nv$.
\end{enumerate}
\end{proposition}

\begin{proof}
Let $v'$ be the least point of $P$ common
to $vNy_0$ and $W_1$.  
If $v'=v$, then $a_1\leq v'=v \leq u \leq b$ in $P$. 
Therefore, $a_1\leq b$ in $P$ and $W_1$ contains $v$ which belongs to $x_0Nv$,
as desired. 
If $v'\neq v$, then $v< v'$ in $P$. 
Let $e'=(w,v')$ be the edge of $W_1$ that occurs immediately before $v'$. 
Then since $W_1$ is $y_0$-left of $W_2$, 
Proposition~\ref{pro:self-evident} implies that 
$e'$ is right of $N$.
If $a_1$ is not right of $N$, 
there exists a vertex $z$ of $a_1W_1w$ which is on $N$.
By the choice of $v'$, $z$ must belong to the $x_0Nv$. 
Therefore, $a_1\leq z\leq u \leq b$ in $P$, so $a_1\leq b$ in $P$, as desired.
\end{proof}

\begin{proposition}\label{pro:force-comparability}
Let $N$ be a separating path associated with $a\leq b$ in $P$. 
If $w<_P z$, $w$ is on one side of $N$ and $z$ is on the other,
then either $w<_P b$ or $a<_P z$.
\end{proposition}

\begin{proof}
Let $W$ be a witnessing path from $w$ to $z$.
Then $W$ and $N$ must intersect.
Let $q$ be a common point.
If $q$ is on the blue part of $N$, then
$w<_P q \leq_P b$.  If $q$ is on the red part of $N$, then
$a\leq_P q < z$.  If $q$ is on the black part of $N$, then
both $w<_P b$ and $a<_P z$ hold.
\end{proof}

Let $z$ be an element of $P$ and $X$ be a subset of elements of $P$. 
We say that
$z$ is \textit{enclosed by $X$} if there is a cycle $D$ in $G$ such that 
(1) all vertices of $D$ are in $X$; and
(2)~$z$ is in the interior of $D$.

\begin{proposition}\label{pro:strict-alternating}
Let $((a_1,b_1),\ldots,(a_k,b_k))$ be a strict alternating cycle of incomparable pairs from $D_P(y_0)\times U_P(x_0)$. 
Also, let $i$, $j$ be distinct integers from $[k]$. 
Then the following statements hold:
\begin{enumerate}
  \item $a_i$ is not enclosed by $U_P(a_j)$;\label{pro:item-a-not-enclosed}
  \item $b_i$ is not enclosed by $D_P(b_j)$;\label{pro:item-b-not-enclosed}
  \item either $a_i$ is $y_0$-left of $a_j$ or $a_i$ is $y_0$-right $a_j$;\label{pro:item-a-left-or-right}
  \item either $b_i$ is $x_0$-left of $b_j$ or $b_i$ is $x_0$-right $b_j$;\label{pro:item-b-left-or-right}
  \item $a_i$ is $y_0$-left of $a_j$ if and only if $b_{i+1}$ is $x_0$-right of $b_{j+1}$ (cyclically).\label{pro:item-a-b-equivalence}
\end{enumerate} 
\end{proposition}
The statement of Propositions~\ref{pro:strict-alternating} and~\ref{pro:2-cycle} are illustrated at Figure~\ref{fig:sac}.

\begin{figure}[h]
\begin{center}
\includegraphics[scale=0.8]{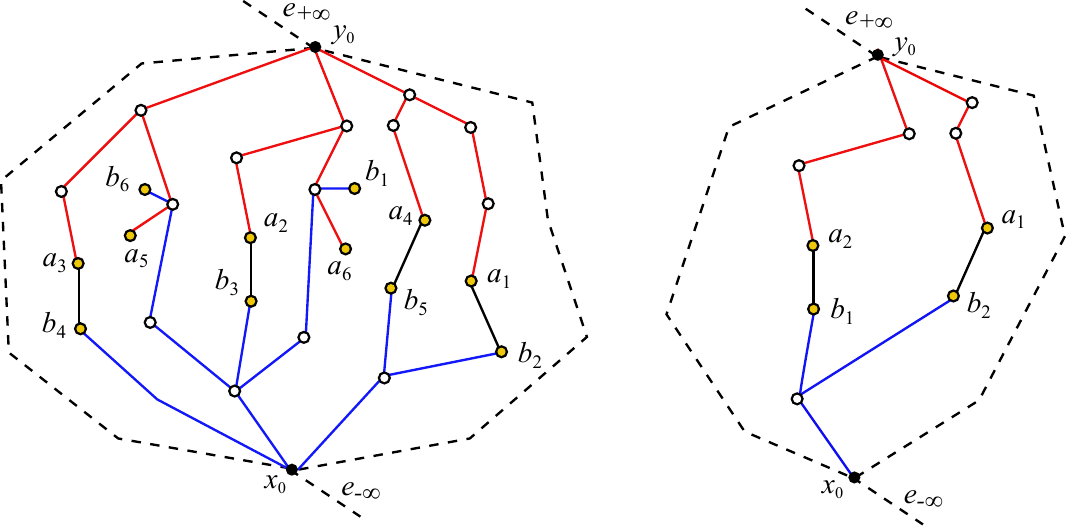}
\end{center}
\caption{Left: A strict alternating cycle $\set{(a_i,b_i)\mid i\in\set{1,\ldots,6}}$ with elements $a_1,a_4,a_6,a_2,a_5,a_3$ ordered from $y_0$-left to $y_0$-right, and $b_2,b_5,b_1,b_3,b_6,b_4$ ordered from $x_0$-right to $x_0$-left. Right: A strict alternating cycle $\set{(a_1,b_1),(a_2,b_2)}$ of length two.}
\label{fig:sac}
\end{figure}

\begin{proof}
Suppose first that $a_i$ is enclosed by $U_P(a_j)$ as 
evidenced by cycle $D\subseteq U_P(a_j)$ in $G$ with 
$a_i$ in the interior of $D$. 
Note also that $x_0$ is not in the interior of $D$ as it is on the exterior face.
Consider a witnessing path $W$ from $x_0$ to $b_{i+1}$. 
If $W$ intersects $D$, 
then a common element $d$ in both shows 
$a_j \leq d \leq b_{i+1}$ in $P$ which contradicts $i\neq j$. 
Thus $D$ and $W$ are disjoint. 
In particular, $b_{i+1}$ is in the exterior of $D$. 
Now consider a witnessing path $W'$ from $a_i$ to $b_{i+1}$. 
Then there is a point $d$ common to $W'$ and $D$. 
Again, this proves $a_j\leq d \leq b_{i+1}$ in $P$. 
The contradiction completes the proof of~\ref{pro:item-a-not-enclosed}. 
The argument for~\ref{pro:item-b-not-enclosed} is dual.

Next we prove statement~\ref{pro:item-a-left-or-right}. 
Let $U_{i+1}$ and $U_{j+1}$ be $x_0$-consistent witnessing paths from $x_0$ to $b_{i+1}$ and $b_{j+1}$, respectively. 
Suppose first that $U_{i+1}$ is $x_0$-left of $U_{j+1}$. 
We claim that $a_i$ is $y_0$-right $a_j$. 
To the contrary, 
suppose $W_i$ and $W_j$ are $y_0$-consistent paths from $a_i$ and $a_j$, respectively, to $y_0$ such that $W_i$ is $y_0$-left of $W_j$. 
Let $N=N(u,v)$ be a separating path associated with $a_i\leq b_{i+1}$ in $P$ such that (1) $u$ is on $U_{j+1}$ and $x_0Nu=x_0U_{j+1}u$; and 
(2) $v$ is on $W_j$ and $vNy_0=vW_jy_0$. 
Such a path exists by Proposition~\ref{pro:sep-path}. 
Since $a_i\parallel b_{j+1}$ in $P$ and 
$W_i$ is $y_0$-left of $W_j$, 
Proposition~\ref{pro:everything} implies that $a_i$ is right of $N$. 
Since $b_{i+1}\parallel a_{j}$ in $P$ and 
$U_{i+1}$ is $x_0$-left of $U_{j+1}$, 
Proposition~\ref{pro:everything} implies that $b_{i+1}$ is left of $N$. 
Since $a_i \leq b_{i+1}$ in $P$, Proposition~\ref{pro:force-comparability} forces 
either $a_i<b_{j+1}$ or $a_j<b_{i+1}$ in $P$. 
Both options are false, which completes the proof that $a_i$ is $y_0$-right of $a_j$. 
Now a symmetric argument shows that if $U_{i+1}$ is $x_0$-right of $U_{j+1}$ then 
$a_i$ is $y_0$-left of $a_j$. 
This completes the proof of item~\ref{pro:item-a-left-or-right}. 
The proof of~\ref{pro:item-b-left-or-right} is dual.

We continue with an argument for~\ref{pro:item-a-b-equivalence}. 
Note that the arguments given for~\ref{pro:item-a-left-or-right} and~\ref{pro:item-b-left-or-right} show that if $b_{i+1}$ is $x_0$-left of $b_{j+1}$, then $a_i$ is $y_0$-right of $a_j$.
Also if $b_{i+1}$ is $x_0$-right of $b_{j+1}$, then $a_i$ is $y_0$-left of $a_j$, as desired.
\end{proof}
The following special case of the preceding proposition is stated for emphasis.
\begin{proposition}\label{pro:2-cycle}
If $((a_1,b_1),(a_2,b_2))$ is a strict alternating cycle of incomparable pairs from $D_P(y_0)\times U_P(x_0)$, then 
$a_1$ is $y_0$-left of $a_2$ if and only if $b_1$ is $x_0$-left of $b_2$.
\end{proposition}

We define an auxiliary digraph $H$ whose vertex set is 
$\Inc(D_P(y_0),U_P(x_0))$. 
When $(a,b)$ and $(a',b')$ are vertices in $H$, 
we have a directed edge from $(a,b)$ to $(a',b')$ in $H$ when 
$((a,b),(a',b'))$ is an alternating cycle, and $a$ is $y_0$-left of $a'$ (therefore, 
$b$ is $x_0$-left of $b'$ by Proposition~\ref{pro:2-cycle}).

When $n\geq1$, a sequence $((a_1,b_1),\ldots,(a_n,b_n))$ of vertices 
from $H$ is a \emph{directed path} of length $n$ in $H$, if 
there is a directed edge in $H$ from $(a_i,b_i)$ to $(a_{i+1},b_{i+1})$ 
for all $i$ such that $1\leq i\leq n-1$. 
Note that when $(a,b)$ is a vertex in $H$, we consider $((a,b))$ as 
a directed path of length one.
Note further that $H$ is acyclic. 

The next proposition implies a notion of transitivity for directed paths in $H$,
and this concept will prove to be fundamentally important.

\begin{proposition}\label{pro:transitive}
Let $n\ge3$ and let $\left((a_1,b_1),\ldots,(a_n,b_n)\right)$ be a 
directed path in $H$.  Then $((a_i,b_i),(a_j,b_j))$ is an edge in 
$H$ for all $i,j$ with $1\le i<j\le n$.
In particular, these pairs form a copy of the standard example $S_n$.
\end{proposition}

\begin{proof}
Using induction, it is clear that the lemma holds in general if
it holds when $n=3$.
Now there are two statements that need to be proved: 
(1) there is an edge in $H$ from $(a_1,b_1)$ to $(a_3,b_3)$; 
(2) the sets $\set{a_1,a_2,a_3}$ and $\set{b_1,b_2,b_3}$ are disjoint. 
This will force a standard example of size $3$ on these six elements.

Now to prove statement (1), we observe first that
$a_1$ is $y_0$-left of $a_2$, and $a_2$ is $y_0$-left of $a_3$. 
Proposition~\ref{pro:y0-left-transitive} implies 
$a_1\parallel a_3$ in $P$, and 
$a_1$ is $y_0$-left of $a_3$. 
It suffices to show that $a_1\leq_P b_3$ and $a_3\leq_P b_1$.
We first show that $a_1\leq_Pb_3$. 

Let $S$ and $T$ be witnessing trees for $\set{a_1,a_2,a_3}$ and $\set{b_1,b_2,b_3}$, respectively. 
Let $N=N(u,v)$ be a separating path associated with $a_2\leq b_3$ in $P$ 
such that (1) $u$ is on $x_0Tb_3$ and $x_0Nu=x_0Tu$; and
(2) $v$ is on $a_2Sy_0$ and $vNy_0=vSy_0$. 
Since $b_2\parallel a_2$ and $b_2$ $x_0$-left of $b_3$, 
Proposition~\ref{pro:everything} implies $b_2$ is left of $N$. 
If $a_1$ is not right of $N$ (see the left part of Figure~\ref{fig:H-transitive}), 
then Proposition~\ref{pro:everything} implies 
$a_1\leq b_3$ in $P$. 
Accordingly, we may assume $a_1$ is right of $N$, 
see the right part of Figure~\ref{fig:H-transitive}.
Since $a_1\leq b_2$ and $a_1$ and $b_2$ are on opposite sides of $N$, 
Proposition~\ref{pro:force-comparability} forces 
either $a_1<b_3$ or $a_2<b_2$ in $P$. 
The second option is false, so we conclude that $a_1<b_3$ in $P$.
The argument for $a_3 \leq_P b_1$ is symmetric, and
this completes the proof of statement (1).

\begin{figure}[h]
\begin{center}
\includegraphics[scale=0.8]{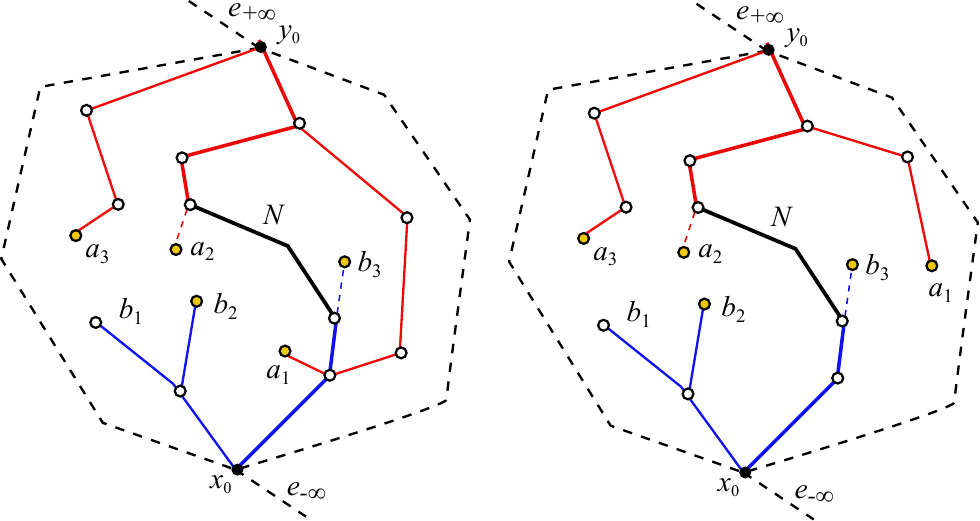}
\end{center}
\caption{Left: $a_1$ is left of $N$ and therefore $a_1\leq b_3$ in $P$. 
Right: $a_1$ is right of $N$ and $b_2$ is left of $N$. 
Thus, a witnessing path from $a_1$ to $b_2$ intersects $N$ and witnesses $a_1\leq b_3$ in $P$ as well.}
\label{fig:H-transitive}
\end{figure}

Now suppose there are integers $i,j\in[3]$ such that 
$a_i=b_j$. 
Let $k$ be the other integer in $[3]$. 
Then $a_k\leq b_j=a_i$ so $a_k$ is comparable to $a_i$, 
which is false. This proves statement (2).
\end{proof}

Let $\mathbb{F}=(P,G,x_0,y_0,I,\mathcal{D})$ be our fixed doubly exposed sequence.
Then let $\rho(\mathbb{F})$ be the maximum size (number of vertices) 
of a directed path in the auxiliary digraph $H$.
The proof of the following lemma is (essentially) the same as
the argument given for Lemma~5.9 in~\cite{ST14}, although 
we are working here in a more general setting.
\begin{lemma}\label{lem:nm}
If $\mathbb{F}=(P,G,x_0,y_0,I,\cgD)$ is doubly exposed, then
\[
\dim(I) \leq \rho(\mathbb{F})^2.
\]
In particular, for $d\geq3$, 
if $d$ is the largest size of a standard example in $P$, 
then $\dim(I)\leq d^2$.
\end{lemma}

\begin{proof}
We show $\dim(I)\leq \rho(\mathbb{F})^2$ by exhibiting a partition of $I$ 
into $\rho(\mathbb{F})^2$ reversible sets.
These sets will have the form $I(m,n)$ where $1\leq m,n \leq \rho(\mathbb{F})$.
A pair $(a,b)\in I$ belongs to $I(m,n)$ if
\begin{enumerate}
\item the longest directed path in $H$ starting from $(a,b)$ has size $m$, and
\item the longest directed path in $H$ ending at $(a,b)$ has size $n$.
\end{enumerate}
To complete the proof, it suffices to show that each $I(m,n)$ 
is reversible.  We argue by contradiction.

Suppose that for some pair $(m,n)$, the set $I(m,n)$ is not reversible. 
Therefore there is a strict alternating cycle 
$((a_1,b_1),\ldots,(a_k,b_k))$ of size $k\geq 2$ with all pairs from $I(m,n)$. 
Fix $S$ and $T$ to be witnessing trees for $\set{a_1,\ldots,a_k}$ and $\set{b_1,\ldots,b_k}$, respectively. 
Without loss of generality, $a_1<_S a_i$ for each $i\in\set{2,\ldots,k}$.

If $k=2$, then there is a directed edge from $(a_1,b_1)$ to $(a_2,b_2)$ in $H$.
It follows that any directed path in $H$ starting at $(a_2,b_2)$ 
can be extended by prepending $(a_1,b_1)$.  Thus $(a_1,b_1)$, $(a_2,b_2)$ 
cannot both belong to $I(m,n)$.
We conclude that $k\geq3$.

The balance of the proof divides into two cases:
\[
a_1<_S a_k <_S a_2 \qquad\text{or}\qquad a_1<_S a_2 <_S a_k.
\]

In the first case, we will show that 
there is a directed path in $H$ of size $m+1$ starting at $(a_1,b_1)$.
In the second case, we will show that
there is a directed path in $H$ of size $n+1$ ending at $(a_2,b_2)$.
Both implications are contradictions.
We will give details of the proof for the first case.  
It will be clear that the argument for the second case is symmetric.

\begin{figure}
\begin{center}
\includegraphics[scale=.8]{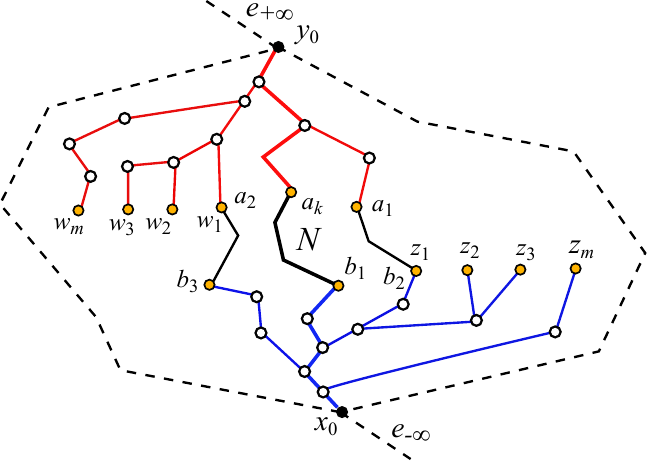}
\end{center}

\caption{The argument shows that $a_2$ is left of $N$ and
$z_2$ is right of $N$.
Therefore, a witnessing path associated with $a_2 \leq_P z_2$ in $P$ has to cross $N$, and this forces $a_k \leq_P z_2$ in $P$.}
\label{fig:cycle2}
\end{figure}

Therefore we assume $a_1<_S a_k <_S a_2$.
Since the pairs $(a_1,b_1)$, $(a_k,b_2)$ form an alternating 
cycle of size $2$ and $a_1 <_S a_k$, 
we have an edge in $H$ from $(a_1,b_1)$ to $(a_k,b_2)$.
Since $(a_1,b_1)$ is the first vertex on this edge, we know $m\geq2$.

Fix a directed path $((w_1,z_1),(w_2,z_2),\ldots,(w_{m},z_m))$ in $H$ 
with $(w_1,z_1) = (a_2,b_2)$. (Recall that $m\geq2$.)
Now consider the sequence
\[
((a_1,b_1),(a_k,b_2),(w_2,z_2),\ldots,(w_{m},z_m)).
\]
We claim that this sequence is a directed path in $H$.
Since it has size $m+1$ and it starts at $(a_1,b_1)$, 
this will be a contradiction.

We have already noted that $((a_1,b_1),(a_k,b_2))$ is an edge in $H$. 
Note that for all $i$ with $2\leq i< m$, $((w_i,z_i),(w_{i+1},z_{i+1}))$ is an edge in $H$ as well. 
It remains only to show that there is an edge from $(a_k,b_2)$ to 
$(w_2,z_2)$ in $H$. 
Since $a_k$ is $y_0$-left of $a_2$ and $a_2=w_1$ is $y_0$-left of $w_2$, 
Proposition~\ref{pro:y0-left-transitive} we know that $a_k$ is $y_0$-left of $w_2$. 
It remains only to show that $a_k\leq z_2$ and $w_2\leq b_2$ in $P$.
Note that 
$w_2 \leq_P z_1 = b_2$.  
Therefore, we only need to show that $a_k \leq_P z_2$.

Let $N=N(u,v)$ be a separating path for $a_k\leq b_1$ in $P$ such that 
(1) $u$ is on $x_0Tb_1$ and $x_0Nu=x_0Tu$; and 
(2) $v$ is on $a_kSy_0$ and $vNy_0=vSy_0$. 
See Figure~\ref{fig:cycle2}. 
Note that Proposition~\ref{pro:strict-alternating}, 
$b_1<_T b_2=z_1 <_T z_2$. 
If $z_2$ is not right of $N$, 
then 
Proposition~\ref{pro:everything} implies $a_k\leq z_2$ in $P$, as desired.
Therefore, we may assume that $z_2$ is right of $N$. 
Since $a_k<_S a_2$ and $a_2\parallel_P b_1$, 
Proposition~\ref{pro:everything} implies $a_2$ is left of $N$. 
Since $a_2=w_1\leq z_2$ in $P$, 
Proposition~\ref{pro:force-comparability} implies 
either $a_2<b_1$ or $a_k<z_2$ in $P$. 
Since the first option is false, 
we must have $a_k<z_2$ in $P$, as desired.
\end{proof}

When $I$ is doubly exposed, we now have $\dim(I)$ bounded in 
terms of $\rho(\mathbb{F})$, \emph{independent} of the height $h$ of~$P$.  
Now we turn our attention to bounding $\rho(\mathbb{F})$ in terms of~$h$.

\section{Restrictions Resulting from Bounded Height}\label{sec:rho-height}

This section is devoted to proving the following lemma.

\begin{lemma}\label{lem:36}
If the sequence $\mathbb{F}=(P,G,x_0,y_0,I,\mathcal{D})$ is doubly exposed, and $P$ is of height at most $h$, 
then
\[
\rho(\mathbb{F}) \leq 58h + 11.
\]
\end{lemma}

Once this lemma has been proven, the proof of our main theorem will be complete.
To see this, recall that using Corollary~\ref{cor:doubly-exposed}, 
we paid a price of $\Oh(h^4)$ to reduce to the case 
where we need to bound $\dim(I)$ for $I$ doubly exposed in $P$.
Lemma~\ref{lem:nm} asserts that $\dim(I) \leq \rho(\mathbb{F})^2$.
Combining this with~Lemma~\ref{lem:36}, we obtain the bound $\Oh(h^6)$.

Our final bound on $\rho(\mathbb{F})$ will emerge from
a series of preliminary results 
all working within the following context. 
We fix a sequence $\mathbb{F}=(P,G,x_0,y_0,I,\cgD)$ which is doubly exposed and let
$h$ be the height of $P$. 
We may assume that $h\geq2$.
We also let $H$ be the auxiliary digraph of $\mathbb{F}$.

The following elementary proposition will play a key role in subsequent
arguments.  There are actually two versions, 
one for $\set{a_1,\ldots,a_n}\subseteq D_P(y_0)$ and one for 
$\set{b_1,\ldots,b_n}\subseteq U_P(x_0)$.
The impact of the proposition
is illustrated in Figure~\ref{fig:no-skip}.

\begin{proposition}\label{pro:no-skip}
Let $((a_1,b_1),\ldots,(a_n,b_n))$ be a directed path in $H$,  
with $n\geq3$, and let $S$ be a witnessing tree for $\set{a_1,\ldots,a_n}$. 
If $1\le i< j< k \le n$ and $W$ is a witnessing path intersecting 
both $a_iSy_0$ and $a_kSy_0$, then 
$W$ intersects $a_jSy_0$.
\end{proposition}

\begin{figure}
\begin{center}
\includegraphics[scale=.8]{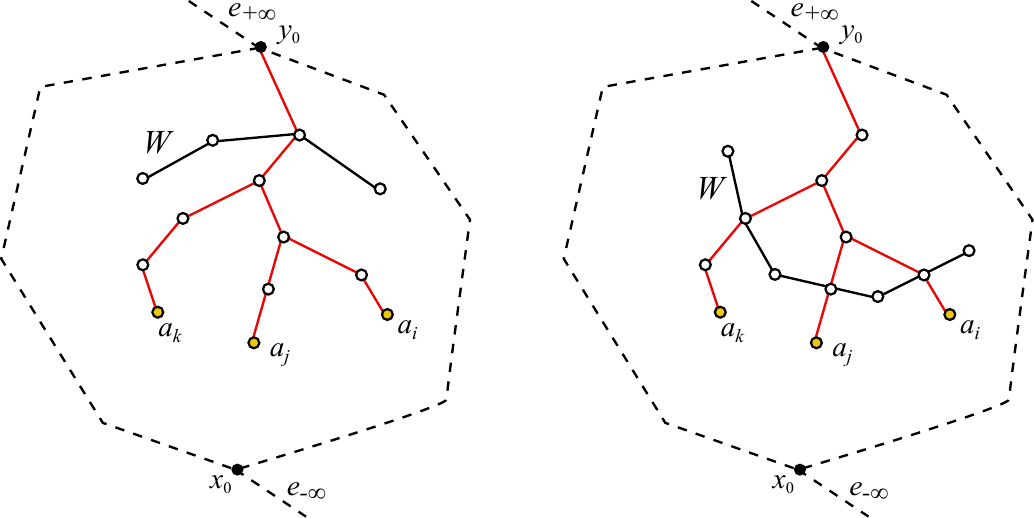}
\end{center}
\caption{A witnessing path $W$ intersects $a_iSy_0$ and $a_kSy_0$, 
therefore it has to intersect $a_jSy_0$.
}
\label{fig:no-skip}
\end{figure}

\begin{proof}
Since $S$ is a tree, if there is an element $w$ of $W$ 
common to $a_iSy_0$ and $a_kSy_0$, 
then $w$ is in $a_jSy_0$. 
Accordingly, we may assume that
(1) $W$ is a non-trivial path intersecting 
both $a_iSy_0$ and $a_kSy_0$ at distinct points 
$s_i$ and $s_k$, respectively; and 
(2) no proper subpath of $W$ intersects both 
$a_iSy_0$ and $a_kSy_0$.

In this case, the paths $W$ and $s_iSs_k$ form a cycle $D$ in $G$. 
If $W$ does not intersect $a_jSy_0$, then by planarity,  
$a_j$ is in the interior of $D$.
Since $W$ is a witnessing path, we have either $s_i < s_k$ or $s_k < s_i$ in $P$. In the first case, $a_i \leq s_i \leq d$ for all $d\in D$. 
This implies that $a_j$ is enclosed by $U_P(a_i)$, 
which is false by 
Proposition~\ref{pro:strict-alternating}. 
A symmetric argument shows that if $s_k<s_i$ in $P$, then 
$a_j$ is enclosed by $U_P(a_k)$. 
The contradiction completes the proof.
\end{proof}

Let $N$ be a separating path associated with a comparability $a\leq b$ in $P$. 
Let $A$ and $B$ be two subsets of elements of $P$.
We will say that $N$ \textit{separates} $A$ from $B$ if 
all points of $A$ are on one side of $N$ and all points of $B$
are on the other side.  

Let $((a_1,b_1),\ldots,(a_n,b_n))$ be a directed path in $H$.
For a non-empty subset $X\subseteq[n]$, we let $A(X)=\{a_i:i\in X\}$ and
$B(X)=\{b_i:i\in X\}$. 
Note that the pairs
in $\{(a_i,b_i):i\in X\}$ determine a directed path of  
size~$|X|$ in $H$. 

We present the first of three key results bounding $\rho(\mathbb{F})$ 
in terms of the height of~$P$.
\begin{proposition}\label{pro:separate}
Let $((a_1,b_1),\ldots,(a_n,b_n))$ be a directed path in $H$.
Let $\alpha$, $\beta$ be distinct integers from $[n]$ and let $N$ be a separating path associated with $a_{\alpha}<b_{\beta}$ in $P$. 
Let $X\subset[n]-\set{\alpha,\beta}$ such that 
either $i<\min(\alpha,\beta)$ for all $i\in X$, 
or $i>\max(\alpha,\beta)$ for all $i\in X$.
If $N$ separates $A(X)$ from $B(X)$, then $|X|\le 2h-1$.
\end{proposition}

\begin{proof}
We give the argument when 
$i<\min(\alpha,\beta)$ for all $i\in X$. 
The argument for the other case is symmetric.

Let $N=N(u,v)$. 
Fix a witnessing tree $S$ for $\set{a_1,,\ldots,a_n}$ such that 
$vNy_0$ is a terminal portion of $a_{\alpha}Sy_0$. 
Also,  
fix a witnessing tree $T$ for $\set{b_1,,\ldots,b_n}$ such that 
$x_0Nu$ is an initial portion of $x_0Tb_{\beta}$. 

Consider the red portion of $N$, 
a chain on at most $h$ elements. 
For each $a\in A(X)$, let $\sigma(a)$ be the lowest element of this chain 
such that $a <_P \sigma(a)$.
Consider also the blue portion of $N$. 
For each $b\in B(X)$, let $\tau(b)$ be the highest element of this chain
such that $\tau(b) <_P b$.

We claim that when $i,j\in X$ and $i<j$, 
then $\sigma(a_j) \leq \sigma(a_i)$ in $P$. 
Consider a witnessing path $W$ from $a_i$ to $\sigma(a_i)$.
The path $W$ intersects $a_iSy_0$ and $a_{\alpha}Sy_0$. 
Proposition~\ref{pro:no-skip} implies that $W$ also intersects $a_jSy_0$ 
(recall that $i<j<\alpha$ by our assumption). 
Therefore, $a_j\leq \sigma(a_i)$ in $P$ which implies $\sigma(a_j)\leq \sigma(a_i)$ in $P$. 
A dual argument shows that when 
$i,j\in X$ and $i<j$, we have $\tau(b_j)\geq \tau(b_i)$ in $P$.

When $i<j$, we claim that at least one of the two inequalities 
$\sigma(a_j) \leq_P \sigma(a_i)$, 
$\tau(b_j)\geq_P \tau(b_i)$ must be strict.
To see this, assume that $\sigma(a_i) = \sigma(a_j)$ and $\tau(b_i) = \tau(b_j)$.
Consider a witnessing path $W_{j,i}$ from $a_j$ to $b_i$.
Since $a_j$ and $b_i$ are on opposite sides of $N$, 
we know that $W_{j,i}$ intersects $N$.
Let $z$ be a common point of $W_{j,i}$ and $N$.
If $z$ is on the red portion of $N$, 
then $b_i > z \geq \sigma(a_j) = \sigma(a_i) 
\geq a_i$ in $P$ which is a contradiction.
If $z$ is on the blue portion of $N$, 
then $a_j < z \leq \tau(b_i) = \tau(b_j) 
\leq b_j$ in $P$ which is a contradiction.
Thus $z$ is in the black part of $N$.
Similarly, a witnessing path $W_{i,j}$ from $a_i$ to $b_j$ must intersect $N$ at a 
point $z'$ which is also in the black part of $N$. 
Since the black part is a chain, $z$ and $z'$ 
are comparable in $P$.  If $z \leq_P z'$, then $a_j \leq z \leq z' \leq b_j$ in $P$.
If $z' \geq_P z$, then $a_i \leq z' \leq z \leq b_i$ in $P$.
Both statements are false.
This observation confirms our claim.

Consider the following two sets $\set{\sigma(a) \mid a\in A(X)}$, 
$\set{\tau(b) \mid b\in B(X)}$.  
Each of these can be considered as a 
sequence sorted by the linear order on $X$ as a set of integers.
The first sequence is non-increasing on the red chain in $N$.
The second sequence is non-decreasing on the blue chain in $N$.
Now moving along elements in $X$ in their natural ordering, 
there are $|X|-1$ consecutive pairs. 
For each such pair, at least one of the two sequences changes. 
Therefore,
\[
|X|-1 \leq |\set{\sigma(a) \mid a\in A(X)}|-1 + |\set{\tau(b) \mid b\in B(X)}|-1 \leq 2(h-1),
\]
so $|X|\leq 2h-1$. 
With this observation, the proof is complete.
\end{proof}

\begin{proposition}\label{pro:disjoint-trees}
Let $((a_1,b_1),\ldots,(a_n,b_n))$ be a directed path in $H$.
If $S$ and $T$ are witnessing trees for $\set{a_1,\ldots,a_n}$ and 
$\set{b_1,\ldots,b_n}$, respectively, and 
$S\cap T$ is empty, then 
$n\leq 6h+1$.
\end{proposition}

\begin{proof}
We assume that $n\ge 6h+2$ and argue to a contradiction.
Let $N=N(u,v)$ be a separating path associated with the comparability $a_{4h}<b_{4h+1}$ such that the red part of $N$ is a terminal portion of $a_{4h}Sy_0$, 
and the blue portion of $N$ is an initial portion of $x_0Tb_{4h+1}$. 
Let $W$ be the black portion of $N$. 
We split the elements of the pairs into
$A_1=\{a_1,a_2,\dots,a_{4h-1}\}$, $A_2=\{a_{4h+2},a_{4h+3},\dots,a_{6h+2}\}$, 
$B_1=\{b_1,b_2,\dots,b_{4h-1}\}$ and $B_2=\{b_{4h+2},b_{4h+3},\dots,b_{6h+2}\}$.

Since $a_{4h+1}\parallel b_{4h+1}$ in $P$,
the path $W$ does not intersect $a_{4h+1}Sy_0$. 
Furthermore, by Proposition~\ref{pro:everything}, 
$a_{4h+1}$ is left of $N$.
Proposition~\ref{pro:no-skip}
implies that if $a\in A_2$, then 
$W$ does not intersect $aSy_0$. 
Since $S$ and $T$ are disjoint, 
Proposition~\ref{pro:everything} now implies
$a$ is left of $N$ as well.
Dually, 
$W$ does not intersect $x_0Tb_{4h}$ and $b_{4h}$ is left of $N$.
Now, Proposition~\ref{pro:no-skip}
implies that if $b\in B_1$, then 
$W$ does not intersect $x_0Tb$, 
and $b$ is left of $N$ as well.
On the other hand,
elements of $A_1\cup B_2$ may be on either side of $N$.

We partition the set $\{1,2,\dots,4h-1\}$
as $X_1\cup X_2$, where $i\in X_1$ if and only if
$a_i$ is left of $N$.  
Since $N$ separates $A(X_2)$ from $B(X_2)\subseteq B_1$, 
it follows from Proposition~\ref{pro:separate}
that $|X_2|\le 2h-1$.  Therefore $|X_1|\geq 4h-1 - (2h-1) = 2h$.
Similarly, we partition $\{4h+2,4h+3,\dots,6h+2\}$
as $Y_1\cup Y_2$, where $i\in Y_1$ if and only if
$b_i$ is left of $N$. 
Now Proposition~\ref{pro:separate} implies that $|Y_2|\le 2h-1$ and therefore $|Y_1|\geq 2h+1 - (2h-1) = 2$.
Set $m=2h$.
Now we are going to discard excess elements and relabel those that remain.
Let $A'$ be a subset of $A(X_1)$ of size $m$ with elements relabeled as
$\set{w_1,\ldots,w_m}$ so that $w_1<_S \cdots <_S w_m$.
Let $B'$ be the corresponding subset of elements of $B(X_1)$ with elements 
relabeled correspondingly as $\set{z_1,\ldots,z_m}$.
Let $\set{z_{m+1},z_{m+2}}$ be a subset of $B(Y_1)$ of size $2$ 
so that $z_{m+1}<_T z_{m+2}$.
Let $\set{w_{m+1},w_{m+2}}$ be the corresponding subset of elements of $A(Y_1)$.
Note that we have
\begin{align*}
w_1 <_S \cdots <_S w_m &<_S a_{4h} <_S a_{4h+1} <_S w_{m+1} <_S w_{m+2},\\
z_1 <_T \cdots <_T z_m &<_T b_{4h} <_T b_{4h+1} <_T z_{m+1} <_T z_{m+2}.
\end{align*}

Let $N'=N(u',v')$ be a separating path associated with $w_{m+1}<z_{m+2}$ 
such that the red part of $N'$ is a terminal portion of $w_{m+1}Sy_0$ and 
the blue part of $N'$ is an initial portion of $x_0Tz_{m+2}$.
Also, let $W'$ denote the black part of $N'$.
Since $z_{m+1}\parallel w_{m+1}$, 
$W$ does not intersect $x_0Tz_{m+1}$. 
Now Proposition~\ref{pro:everything} implies that $z_{m+1}$ is left of 
$N'$. 
Using Proposition~\ref{pro:no-skip}, it follows 
that if $b\in B'$, then $W'$ does not intersect $x_0Tb$, and $b$ is left of $N'$.  

\smallskip
\noindent
\textbf{Claim.}\ 
All elements of $A'$ are right of $N'$.
\begin{proof}
Consider an element $a\in A'$.
Since $a$ is left of $N$, $aSy_0$ contains a point from the union 
of the black and blue parts of $N$. 
Since $S\cap T$ is empty, $W$ intersects $aSy_0$. 
Let
$p$ be the largest point of $W$ that is also on $aSy_0$.
Dually, since $z_{m+2}$ is left of $N$, 
we know that $W$ intersects $x_0Tz_{m+2}$. 
Let $q$ be the least element of $W$ that is also 
on $x_0Tz_{m+2}$. 
Since $S$ and $T$ are disjoint and by planarity, 
we know $v\leq p< q \leq u$ in $P$.
Clearly, $q$ is the \emph{first} point of $x_0Tz_{m+2}$ 
that lies in $W$, and $p$ is the \emph{last} point of $aSy_0$ 
that lies in $W$.

Proposition~\ref{pro:no-skip} implies that there
is a point $r$ common to $qWu$ and $x_0Tz_{m+1}$.
In particular, we have $q \leq r \leq z_{m+1}$ in $P$.

Recall that $W'$ is a witnessing path from $v'$ to $u'$, 
and $u'$ is in $x_0Tz_{m+2}$.
If $u'$ is on $x_0Tq$ 
then
\[
w_{m+1} < u'\le q\le r \leq z_{m+1} \text{ in $P$},
\]
which is a contradiction. 
We conclude that 
$u'$ is an element of $x_0Tz_{m+2}$ that occurs \emph{after}
$q$. 
Let $e$ be the first edge of $qTz_{m+2}$ and 
let $e_0$ be the last edge of $x_0Tq$.
Now let $e'$ be the first edge $qWp$ (it does exist as $q\neq p$). 
See Figure~\ref{fig:pillar}.
We assert that $e$ is left of $e'$ in the $(q,e_0)$-ordering.
To verify this assertion, 
consider the path $N''=x_0TqWvSy_0$. 
Since $z_{m+2}$ is left of $N$, it is left of $N''$ as well. 
Now suppose to the contrary that 
$e'$ is left of $e$ in the $(q,e_0)$-ordering. 
Then $e$ is right of $N''$. 
Now Proposition~\ref{pro:self-evident} implies that 
there is a vertex $q'$, with $q'\neq q$ such that 
$q'$ is common to $N''$ and $qTz_{m+2}$. 
Since $T$ is a tree, $q'$ cannot be on $x_0N''q=x_0Tq$. 
Since $S$ and $T$ are disjoint, $q'$ cannot be on $vN''y_0=vSy_0$. 
Therefore, $q'$ is on $qN''v=qWv$.
This contradicts the choice of $q$ as the least element 
common to $W$ and $x_0Tz_{m+2}$.
This observation completes the proof of the assertion.

\begin{figure}
\begin{center}
\includegraphics[scale=.8]{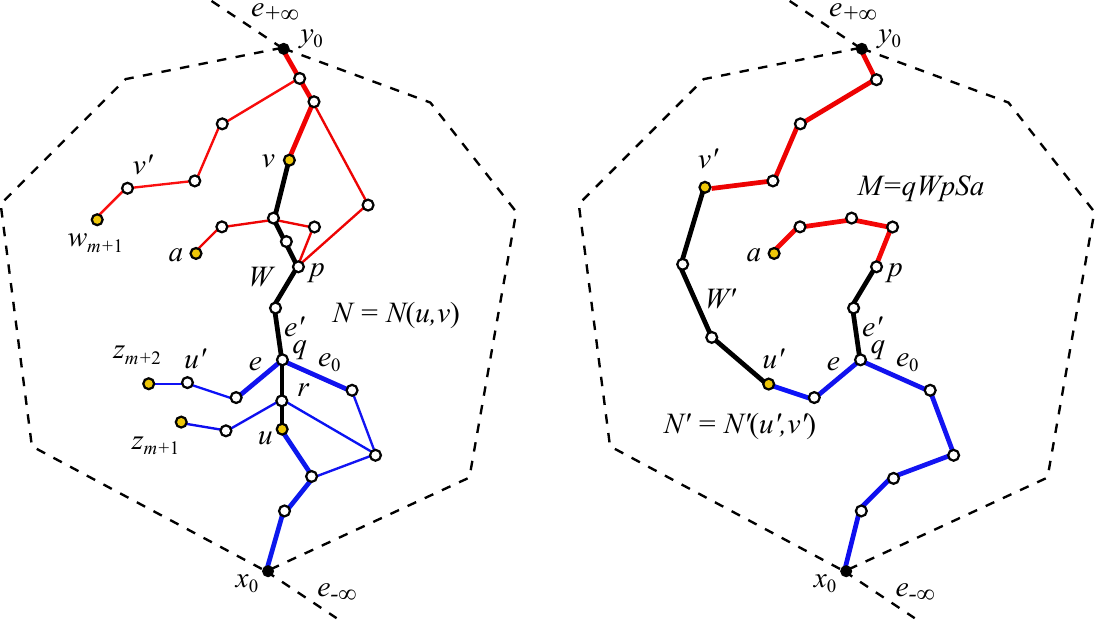}

\end{center}
\caption{The edge $e$ is left of $e'$ in the $(q,e_0)$-ordering.  
This makes $e'$ right of $N'$ and therefore $a$ is right of $N'$.}
\label{fig:pillar}
\end{figure}

Note that by the definition of $p$ and $q$, 
$M=qWpSa$ is a path.
Recall that $e$ in on $N'$ and $e'$ is on $M$.
Since $e$ is left of $e'$ in the $(q,e_0)$-ordering, 
Proposition~\ref{pro:self-evident} implies that $e'$ is right of $N'$.

To complete the proof, 
we observe that if $a$ is not right of $N'$, then 
Proposition~\ref{pro:self-evident} implies that 
there is a vertex $q'$, with $q'\neq q$, 
common to $M$ and $N'$.
If $q'$ is in the black or red part of $N'$, i.e.\ $u'N'y_0$, 
then $v'\leq q'$ in $P$ and
\[
w_{m+1} \leq v' \leq q' \leq q \leq r \leq z_{m+1}\ \text{in $P$,} 
\]
a contradiction.
If $q'$ is in the blue part of $N'$, i.e.\ $x_0Tu'$, then 
$q'$ cannot belong to $aMp=aSp$ as $S$ and $T$ are disjoint.
Finally, if $q'$ is in the blue part of $N'$ and 
$q'$ is in $pMq=pWq$, then we contradict the choice of $q$.
\end{proof}

We have now reached a
contradiction since we have shown that $N'$ separates $A'$ and $B'$ with
$|A'|=|B'|=m=2h$, contradicting Proposition~\ref{pro:separate}.
This completes the proof of Proposition~\ref{pro:disjoint-trees}.
\end{proof}

\subsection{Separating witnessing trees}

In Figure~\ref{fig:vertical}, we illustrate 
some of the challenges we face in finding disjoint 
witnessing trees.

\begin{figure}
\begin{center}
\includegraphics[scale=.7]{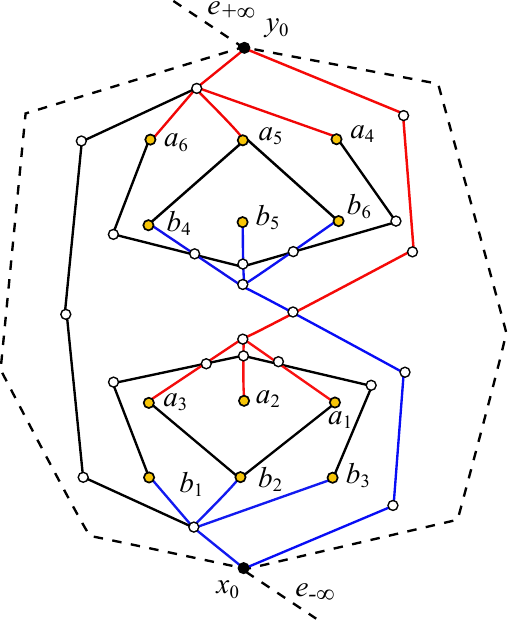}
\hspace{.2in}
\includegraphics[scale=.7]{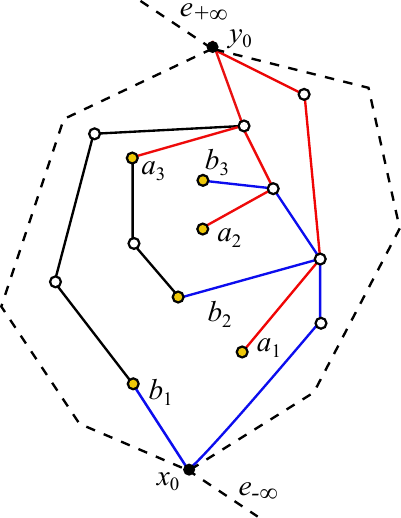}
\end{center}
\caption{
Left: Two copies of $S_3$ are stacked in a vertical manner. 
Clearly, this construction can be expanded for two copies of an 
arbitrarily large standard example. 
Right: Three incomparable pairs are stacked vertically to form 
a copy of $S_3$.}
\label{fig:vertical}
\end{figure}

Now we begin the material to address this challenge. 
Let $Z$ be the non-empty subposet of $P$
consisting of all elements of $P$ that belong to a witnessing path from
$x_0$ to $y_0$.  If we restrict our drawing of $G$ to
the induced subgraph determined by the elements of $Z$, we obtain
a drawing without edge crossings of the cover graph of $Z$.
Furthermore, $x_0$ is the unique minimal element of $Z$,
and $y_0$ is the unique maximal element of $Z$, and 
the induced subgraph of $G$ determined by the elements of $Z$ 
is the cover graph of the subposet $Z$. 

Our fixed drawing of the cover graph of $Z$ splits the plane into regions: 
some number of bounded regions and one unbounded. 
We call such a bounded region a $Z$-\emph{face}.
Each element of $P$ that is not in $Z$ is in the interior of one of the regions.
After adding two dummy elements $z'$, $z''$ into $P$ (and $Z$) such that
(1) $x_0 < z' < y_0$, $x_0 < z'' < y_0$ in $P$ and all these relations are covers; 
(2) $x_0z'$ is the leftmost edge in the $(x_0,e_{-\infty})$-ordering; 
(3) $x_0z''$ is the rightmost edge in the $(x_0,e_{-\infty})$-ordering; 
we can assume that any element of $P$ that is not in $Z$ is in the interior
of one of the (bounded) $Z$-faces.  

Let $S(\textrm{left},\textrm{not-left})$ 
consist of all pairs $(a,b)\in \Inc(D_P(y_0),U_P(x_0))$ for which 
there exists a witnessing path from $x_0$ to $y_0$ such that 
$a$ is left of $W$ and $b$ is not left of $W$.
Similarly, 
let $S(\textrm{not-left},\textrm{left})$ consist of all pairs 
$(a,b)\in \Inc(D_P(y_0),U_P(x_0))$ for which 
there exists a witnessing path from $x_0$ to $y_0$ such that 
$a$ is not left of $W$ and $b$ is left of $W$.
The other two sets
$S(\textrm{right},\textrm{not-right})$ and
$S(\textrm{not-right},\textrm{right})$ are defined in a symmetric manner.

\begin{proposition}
\label{pro:michal}

Each of the sets 
$S(\textrm{left},\textrm{not-left})$, 
$S(\textrm{right},\textrm{not-right})$, 
$S(\textrm{not-left},\textrm{left})$, 
$S(\textrm{not-right},\textrm{right})$ 
is reversible.
\end{proposition}

\begin{proof}
Using an argument by contradiction, 
we show that $S(\textrm{left},\textrm{not-left})$ is reversible. 
The argument for the other three sets is symmetric.
Let $((a_1,b_1),\ldots,(a_k,b_k))$ be a strict alternating cycle in $S(\textrm{left},\textrm{not-left})$.
Let $i\in[k]$. 
Fix a witnessing $W$ from $x_0$ to $y_0$ such that 
$a_i$ is left of $W$ and $b_i$ is not left of $W$. 
If $b_i$ is on $W$ then let $W_i=x_0Wb_i$. 
Otherwise,
let $W_{i}$ be a witnessing path from $x_0$ to $b_i$ that is 
$x_0$-consistent with $W$, so we have that 
$W_i$ is $x_0$-right of $W$.

If $b_{i+1}$ is left of $W$, then 
let $W_{i+1}$ be a witnessing path from $x_0$ to $b_{i+1}$ that is 
$x_0$-consistent with $W$. 
Clearly, $W_{i+1}$ is $x_0$-left of $W$.
Since $b_{i+1}\parallel b_i$ in $P$ we conclude that
$W_{i+1}$ and $W_i$ are $x_0$-consistent and 
$W_{i+1}$ is $x_0$-left of $W_i$.
Thus in this case, Proposition~\ref{pro:strict-alternating} implies that 
$b_{i+1}$ is $x_0$-left of $b_i$.

Now consider the case when $b_{i+1}$ is not left of $W$, and 
let $W'$ be a witnessing path from $a_i$ to $b_{i+1}$. 
Since $a_i$ is left of $W$ and $b_{i+1}$ is not left $W$, 
we conclude that 
$W'$ and $W$ must intersect, say at element $z$. 
Note that $z$ is not in $W_i$, 
as this would imply $a_i < z < b_i$ in $P$, which is false. 
Let $W_{i+1}=x_0WzW'b_{i+1}$. 
Then $W_{i+1}$ and $W_i$ are $x_0$-consistent and $W_{i+1}$ is 
$x_0$-left of $W_i$. 
Again, we conclude by Proposition~\ref{pro:strict-alternating} that 
$b_{i+1}$ is $x_0$-left of $b_i$.

We have now shown that $b_{i+1}$ is $x_0$-left of $b_i$. 
Clearly, this statement cannot hold for all $i\in[k]$. 
The contradiction completes the proof.
\end{proof}

Let $J$ be the subset of pairs in $\Inc(D_P(y_0),U_P(x_0))$ that 
do not belong to any of 
$S(\textrm{left},\textrm{not-left})$, 
$S(\textrm{not-left},\textrm{left})$,  
$S(\textrm{right},\textrm{not-right})$, 
$S(\textrm{not-right},\textrm{right})$.
Note that when $(a,b)\in J$ then neither $a$ nor $b$ belongs to $Z$.
Proposition~\ref{pro:michal} implies
\[
\dim(I) \leq \dim(\Inc(D_P(y_0),U_P(x_0))) \leq \dim(J) +4.
\]

Each $Z$-face $\cgF$ is bounded by two distinct witnessing paths
that have only their starting and ending points in common.
We let $x_\cgF$ denote the common starting point,
and we let $y_\cgF$ denote the common ending point of these two witnessing paths.
When we start at $x_\cgF$ and traverse
the boundary of $\cgF$ in a clockwise manner, we follow
the \emph{left side} of $\cgF$ until we reach $y_\cgF$. 
Then we traverse the \emph{right side} of $\cgF$
backwards until we arrive back at $x_\cgF$.
An element on the left (right) side of $\cgF$ that is not in $\set{x_\cgF,y_\cgF}$ is said to be \emph{strictly} on the left (right) side of $\cgF$. 
Note that if $u$ is strictly on the left side of $\cgF$, and 
$v$ is strictly on the right side of $\cgF$, then 
$u\parallel v$ in $P$. 
Note that $u$ is $x_0$-left of $v$ and $y_0$-right of $v$.
Note also that there is always at least one point strictly on the left (right) side of $\cgF$. 

When $\cgF$ is a $Z$-face, no element $u$ of $P$ that is in the 
interior of $\cgF$ satisfies $x_\cgF<_P u<_P y_\cgF$; otherwise this region 
would be split into smaller $Z$-faces. Also, a $Z$-face has no chords.

When $u\in P$ and $u$ is not in $Z$, there is a unique $Z$-face
$\cgF_u$ containing $u$ in its interior.  We let $y_u=y_{\cgF_u}$ and
$x_u=x_{\cgF_u}$.   
Let $(a,b)\in J$. 
A witnessing path from $a$ to $y_0$ has to leave the interior of $\cgF_a$,
and this implies $a < y_a$ in $P$. 
Dually, a witnessing path from $b$ to $x_0$ (going backward) has to 
leave the interior of $\cgF_b$, and this implies 
$x_b < b$ in $P$.

A pair $(a,b)\in J$ is called  a \textit{same-face pair}
if $\cgF_a=\cgF_b$.

\begin{proposition}
\label{pro:same-face}
All pairs in $J$ are same-face pairs.
\end{proposition}

\begin{figure}
\begin{center}
\includegraphics[scale=.8]{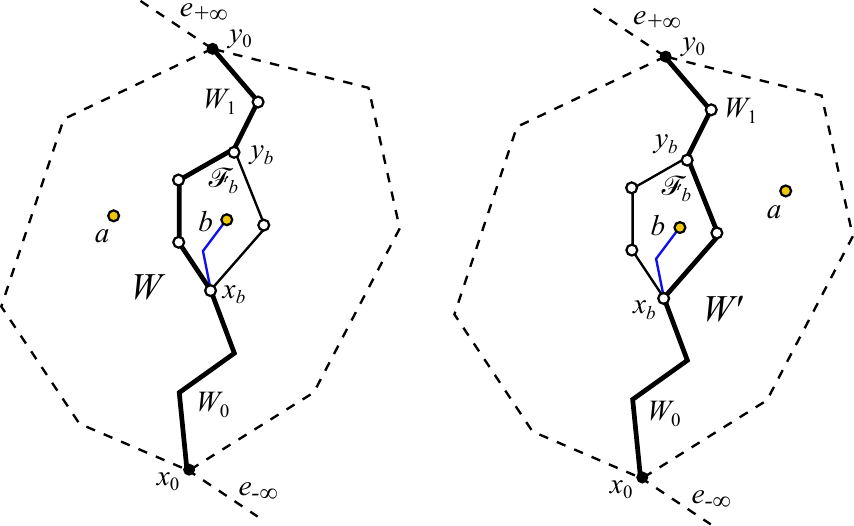}

\end{center}
\caption{An incomparable pair $(a,b)$ with $a$ and $b$ 
in different $Z$-faces. 
Left: $a$ is left of $W$ and $b$ is right of $W$. 
Right: $a$ is right of $W'$ and $b$ is left of $W'$.}
\label{fig:diff-face}
\end{figure}

\begin{proof}
To the contrary, suppose $(a,b)\in J$ and 
$\cgF_a\neq \cgF_b$. 
Let $W_0$ be a witnessing path from $x_0$ to $x_b$, and 
let $W_1$ be a witnessing path from $y_b$ to $y_0$. 
Now let $W$ be the witnessing path from $x_0$ to $y_0$ formed by 
concatenating $W_0$, the left side of $\cgF_b$, and $W_1$. 
Also let $W'$ be the witnessing path from $x_0$ to $y_0$ formed by 
concatenating $W_0$, the right side of $\cgF_b$, and $W_1$. 
See Figure~\ref{fig:diff-face}.

The elements in the interior of $\cgF_b$ are the only points in the plane 
that are right of $W$ and left of $W'$. 
It follows that either (1) $a$ is left of $W$ and $b$ is right of $W$; 
or (2) $a$ is right of $W'$ and $b$ is left of $W'$. 
If (1) holds then $(a,b)\in S(\textrm{left},\textrm{not-left})$, and 
if (2) holds then $(a,b)\in S(\textrm{right},\textrm{not-right})$.
\end{proof}

Here is another self-evident proposition stated for emphasis.

\begin{proposition}\label{pro:z-face}
If $W$ is a witnessing path from $x_0$ to $y_0$ in $Z$ 
and $\cgF$ is a $Z$-face, then 
there do not exist points $u$, $v$ on the boundary of $\cgF$ such that 
$u$ is left of $W$ and $v$ is right of $W$.
\end{proposition}

\begin{proposition}
\label{pro:zero}
Let $u$, $v$ be incomparable elements of $Z$.
\begin{enumerate}
  \item Then either $u$ is $x_0$-left of $v$, 
  or $u$ is $x_0$-right of $v$.
  \item If $W$ is a witnessing path from $x_0$ to $y_0$ 
with $v\in W$.
Then, $u$ is $x_0$-left of $v$ if and only if 
$u$ is left of $W$. 
\end{enumerate}
\end{proposition}
\begin{proof}
Note that $((u,v),(v,u))$ is 
a strict alternating cycle of elements in $Z$. 
Proposition~\ref{pro:strict-alternating} implies that 
either $u$ is $x_0$-left of $v$ or $u$ is $x_0$-right of $v$. 

Let $W$ be a witnessing path from $x_0$ to $y_0$ 
with $v\in W$. 
Suppose $u$ is $x_0$-left of $v$ and let $W_u$ be 
a witnessing path from $x_0$ to $u$ such that 
the pair $(W_u,x_0Wv)$ is $x_0$-consistent. 
Proposition~\ref{pro:self-evident} implies that 
$u$ is left of $W$.

Conversely, suppose that $u$ is left of $W$, and let 
$z$ be the largest element of $W$ such that $z < u$ in $P$. 
Let $U$ be a witnessing path from $z$ to $u$. 
Then $x_0WzUu$ and $x_0Wv$ are $x_0$-consistent. 
Again by Proposition~\ref{pro:self-evident}, the first path is $x_0$-left of the second. 
This shows $u$ is $x_0$-left of $v$.
\end{proof}

\begin{proposition}\label{pro:J}
For every strict alternating cycle $((a_1,b_1),\dots,(a_k,b_k))$ 
of pairs in $J$, there is a $Z$-face $\cgF$ such that 
all elements $a_1,\ldots, a_k, b_1,\ldots,b_k$ are in the interior 
of $\cgF$.
\end{proposition}

\begin{proof}
We assume to the contrary that $((a_1,b_1),\dots,(a_k,b_k))$
is a strict alternating cycle of pairs from $J$, 
and there is
no $Z$-face that contains all elements of the cycle in its
interior.  Of all such cycles, we assume further that $k$ is minimum.  

\noindent
\textbf{Claim 1.}\quad
There do not exist distinct integers $i,j\in[k]$ such that
the pairs $(a_i,b_i)$ and $(a_j,b_j)$ are in the
same $Z$-face.

\begin{proof}
Suppose that for some $i\neq j$ all four elements involved in 
$(a_i,b_i)$, $(a_j,b_j)$ lie in the same $Z$-face.
Since our alternating cycle is a counterexample, we do not have all 
the pairs lying in the same $Z$-face, so we know that $k\ge3$.  After a relabeling, 
we may assume that $j=k$ and $2\le i\le k-1$.  However, this implies that
\[
((a_1,b_1),\dots,(a_{i-1},b_{i-1}),(a_k,b_i))
\]
is an alternating cycle of same-face pairs from $J$.
This is
a contradiction unless all the pairs on this cycle belong the same
$Z$-face.  In this case, we consider
the strict alternating cycle
\[
((a_i,b_1), (a_{i+1},b_{i+1}),\dots,(a_k,b_k)).
\]
Again, we have a strict alternating cycle 
of pairs from $J$. 
However, now it is
clear that not all the pairs on this cycle belong to the same
$Z$-face.  Furthermore, the length of this cycle is less than
$k$.  The contradiction completes the proof of the claim.
\end{proof}
 
For each $i\in[k]$, let $\cgF_i$ be the common $Z$-face $\cgF_{a_i}=
\cgF_{b_i}$, let $x_i=x_{b_i}$, and let $y_i=y_{a_i}$.
Let $W_i$ be a witnessing path from $a_i$ to $b_{i+1}$.  Then
let $u_i$ be the lowest point of $W_i$ that is on the
boundary of $\cgF_i$, and let $v_{i+1}$ be the highest point of $W_i$
that is on the boundary of $\cgF_{i+1}$. 
We note that $a_i<_P u_i\le_P v_{i+1}<_P b_{i+1}$.

\smallskip
\noindent
\textbf{Claim 2.}\quad
For all $i,j\in[k]$, $u_i\le_P v_j$ if and only if $j=i+1$ (cyclically).

\begin{proof}
We already know that $u_i\le_P v_{i+1}$ for all $i\in[k]$.  Now
suppose $j\neq i+1$ and $u_i\le v_j$.  Then $a_i<_P u_i\le_P v_j<b_j$.
This implies $a_i<b_j$.  Now we have contradicted the assumption that
our original cycle is strict.  With this observation, the proof
of the claim is complete.
\end{proof}

Claim 2 implies that $((u_1,v_1),\dots,(u_k,v_k))$ is a strict alternating
cycle of incomparable pairs in $Z$.
Let $i\in[k]$.  Since $u_i\parallel_P v_i$, and both
$u_i$ and $v_i$ are on the boundary of $\cgF_i$, it implies that they
are on opposite sides of $\cgF_i$.  Also, $\{u_i,v_i\}\cap
\{x_i,y_i\}=\emptyset$.  

\smallskip
\noindent
\textbf{Claim 3.}\quad
For each $i\in[k]$, the following statements hold.
\begin{enumerate}
\item If $u_i$ is $x_0$-left of $v_i$, 
then $u_{i+1}$ is $x_0$-left of $v_{i+1}$ and 
$u_{i+1}$ is $x_0$-left of $u_i$.
\item If $u_i$ is $x_0$-right of $v_i$, 
then $u_{i+1}$ is $x_0$-right of $v_{i+1}$ and 
$u_{i+1}$ is $x_0$-right of $u_i$.
\end{enumerate}

\begin{figure}
\begin{center}
\includegraphics[scale=.8]{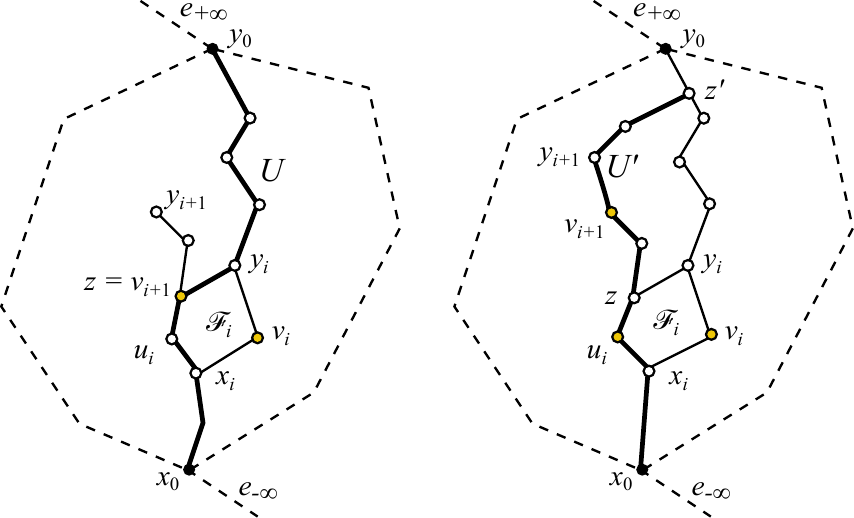}

\end{center}
\caption{
We assume $u_i$ is $x_0$-left of $v_i$. 
Left: $v_{i+1}=z$, an element strictly on the left side of $\cgF_i$.
Right: $z<v_{i+1}$ in $P$, so 
$v_{i+1}$ is left of $U$.  
}
\label{fig:uivi-consistent}
\end{figure}

\begin{proof}
We prove the first statement.  The proof
of the second is symmetric.  Let $i\in[k]$. 
Suppose that $u_i$ is $x_0$-left of $v_i$.
This implies that $u_i$ is strictly on the left side of $\cgF_i$ 
and $v_i$ is strictly on the right side of $\cgF_i$. 

Let $W_0$ be a witnessing path from $x_0$ to $x_i$ and 
let $W_1$ be a witnessing path from $y_i$ to $y_0$.
Let $U$ be the concatenation of $W_0$, 
the left side of $\cgF_i$, and $W_1$.

Recall that $u_i\leq v_{i+1}$ in $P$.
Let $z$ be the largest element on $U$ such that 
$z\leq v_{i+1}$ in $P$.
Since $v_{i+1}\parallel v_i$ in $P$, 
we must have $z$ in $u_iUy_i$ but excluding $y_i$, 
i.e., $z$ is strictly on the left side of $\cgF_i$. 
We split the proof into two cases: 
$z=v_{i+1}$ and $z<v_{i+1}$ in $P$.

As illustrated on the left side of Figure~\ref{fig:uivi-consistent}, 
we first consider the case that $z=v_{i+1}$.
Since $\cgF_i\neq \cgF_{i+1}$, 
$v_{i+1}$ must be strictly on the right side of $\cgF_{i+1}$. 
Therefore, $u_{i+1}$ must be strictly on the left side of $\cgF_{i+1}$, 
$u_{i+1}$ is $x_0$-left of $v_{i+1}$, and 
$u_{i+1}$ is left of $U$.
Proposition~\ref{pro:zero} implies that
$u_{i+1}$ is $x_0$-left of $u_{i}$, as desired.

Now we consider the case that $z<v_{i+1}$ in $P$, 
as illustrated on the right side of Figure~\ref{fig:uivi-consistent}.
Let $W$ be a witnessing path from $z$ to $v_{i+1}$. 
By the definition of $z$, $W$ is disjoint from $U$ except 
the element $z$. 
Note also that the first edge of $W$ is left of $U$, 
as no edge of a witnessing path between two elements in $Z$ 
is in the interior of a $Z$-face.
Proposition~\ref{pro:self-evident} implies that 
all edges and vertices of $W$ except $z$ are left of $U$. 
Let $z'$ be the least element on $U$ such that 
$v_{i+1}<z'$ in $P$. 
Let $W'$ be a witnessing path from $v_{i+1}$ to $z'$.

Now let $U'=x_0UzWv_{i+1}W'z'Uy_0$. 
We claim that $u_{i+1}$ is left of $U'$. 
Clearly, since $u_{i+1}\parallel v_{i+1}$ in $P$, 
$u_{i+1}$ is not on $U'$. 
Suppose to the contrary that $u_{i+1}$ is right of $U'$. 
Recall that $v_{i+1}$ is left of $U$. 
Proposition~\ref{pro:z-face} implies that $u_{i+1}$ is not right of $U$. 
Since $u_{i+1}\parallel u_{i}$ in $P$, 
we conclude $u_{i+1}$ is left of $U$.
Note that all points in the plane that are right of $U'$ and left of $U$ are in the interior of the cycle $zWv_{i+1}W'z'Uz$. 
However, all elements of the cycle are in $U_P(u_i)$. 
This means that $u_{i+1}$ is enclosed by the $U_P(u_i)$, 
a contradiction.
This proves that $u_{i+1}$ is left $U'$.
Now, Proposition~\ref{pro:zero} implies that 
$u_{i+1}$ is $x_0$-left of $v_{i+1}$, and 
$u_{i+1}$ is $x_0$-left of $u_{i}$, as desired.
\end{proof}

To complete the proof of the lemma, 
we simply note that the statement of the claim cannot hold for all $i\in[k]$ cyclically.
\end{proof}

Let $\cgF$ be a $Z$-face. 
We define $P_{\cgF}$ as the subposet of $P$ containing 
all elements inside or on the boundary of $\cgF$.
Note that the cover graph $G_\cgF$ of $P_\cgF$ is 
the induced subgraph of $G$ on elements of $P_\cgF$. 
Let $J_\cgF$ consist of those pairs $(a,b)\in J$ such
that $\cgF=\cgF_a=\cgF_b$. 
Finally, let $\cgD_\cgF$ be the restriction of the plane drawing $\cgD$ to vertices and edges $G_\cgF$.
Observe that $\mathbb{F}_\cgF=(P_\cgF,G_\cgF,x_{\cgF},y_{\cgF},J_\cgF,\cgD_\cgF)$ is doubly exposed.
Then, Propositions~\ref{pro:michal}, \ref{pro:same-face}, \ref{pro:J} imply
\[
\rho(\mathbb{F}) \leq 4+\max_{\cgF}\ \rho(\mathbb{F}_{\cgF}).
\]

\begin{proposition}\label{pro:killer}
Let $\mathbb{F}=(P,G,x_0,y_0,I,\cgD)$ be doubly exposed, 
where $P$ is a poset of height $h$. 
Suppose further that:
\begin{enumerate}
\item The boundary of the exterior face of $G$ in $\cgD$ 
is the boundary of a $Z$-face $\cgF$.
\item If $(a,b)\in I$, 
then both $a$ and $b$ are in the interior of $\cgF$.
\end{enumerate}
Then
\[\rho(\mathbb{F})\leq 58h+7.\]
\end{proposition}

\begin{proof}
Suppose to the contrary that $((a_1,b_1),\dots,(a_n,b_n))$ 
is a directed
path in the auxiliary graph $H$ of $\mathbb{F}$ 
with $n\geq 58h+8$. 
Let $S$ and $T$ be witnessing trees for $\set{a_1,\ldots,a_n}$ and 
$\set{b_1,\ldots,b_n}$, respectively. 
Suppose that $(\alpha,\beta)$ is a pair of distinct integers in $[n]$
such that 
$x_0Tb_\beta$ intersects $y_0Sa_\alpha$. 
Let $c=c(a_\alpha,b_\beta)$ and $d=d(a_\alpha,b_\beta)$ be,
respectively, the least and the greatest element of $P$ common to
$x_0Tb_\beta$ and $y_0Sa_\alpha$. 
Then $x_0< c\le d< y_0$ in $P$. 
Since the only witnessing paths from $x_0$ to $y_0$ are
the two sides of $\cgF$, 
it follows that $W=x_0TcTdSy_0=x_0TcSdSy_0$ is one of the sides of $\cgF$.

If $\alpha<\beta$, we assert (as illustrated on the left side of Figure~\ref{fig:killer}) that the following three statements hold:
\begin{enumerate}
\item
$W$ is the right side of $\cgF$.
\item 
If $1\leq i< \alpha$, then $cS y_0$ is a terminal
portion of $a_{i}Sy_0$. Furthermore,
all edges and vertices of $a_iSc$, except $c$, 
are in the interior of $\cgF$.
\item
If $\beta<j\le n$, then $x_0Td$ is an initial portion
of $x_0Tb_{j}$. 
Furthermore,
all edges and vertices of $dTb_j$, except $d$, 
are in the interior of $\cgF$.
\end{enumerate}

To verify this assertion, suppose first that 
$\alpha<\beta$ and $W$ is the left side of $\cgF$. 
Recall that $a_{\alpha}$ is $y_0$-left of $a_{\beta}$. 
Since $c\in a_{\alpha}Sy_0$ and $c$ is on the left side of $\cgF$, 
it follows that $c\in a_{\beta}Sy_0$. 
This forces $a_{\beta}<c<b_{\beta}$ in $P$, which is false.
The contradiction shows that $W$ is on the right side of $\cgF$. 

If $1\leq i<\alpha$, then $a_i$ is $y_0$-left of $a_{\alpha}$. 
Since $c\in a_{\alpha}Sy_0$ and $c$ is on the right side of $\cgF$, 
it follows that $c\in a_{i}Sy_0$. 
Therefore $cSy_0$ is a terminal portion of $a_iSy_0$.
The definition of $c$ implies that the first edge of $cSa_i$ is 
in the interior of $\cgF$. 
Since $\cgF$ is a $Z$-face, all edges and vertices of $cSa_i$, 
except $c$, are in the interior of $\cgF$.
A symmetric argument shows that if $\beta<j\leq n$, 
then $x_0Td$ is an initial portion of $x_0Tb_j$, and 
all edges and vertices of $dTb_j$, except $d$, 
are in the interior of $\cgF$.

Symmetrically, if $\alpha>\beta$, 
the following three statements hold:
\begin{enumerate}
\item
$W$ is the left side of $\cgF$.
\item 
If $\alpha<i\leq n$, then $cS y_0$ is a terminal
portion of $a_{i}Sy_0$. 
Furthermore,
all edges and vertices of $a_iSc$, except $c$, 
are in the interior of $\cgF$.
\item
If $1\leq j<\beta$, then $x_0Td$ is an initial portion
of $x_0Tb_{j}$. 
Furthermore,
all edges and vertices of $dTb_j$, except $d$, 
are in the interior of $\cgF$.
\end{enumerate}

\begin{figure}
\begin{center}
\includegraphics[scale = .8]{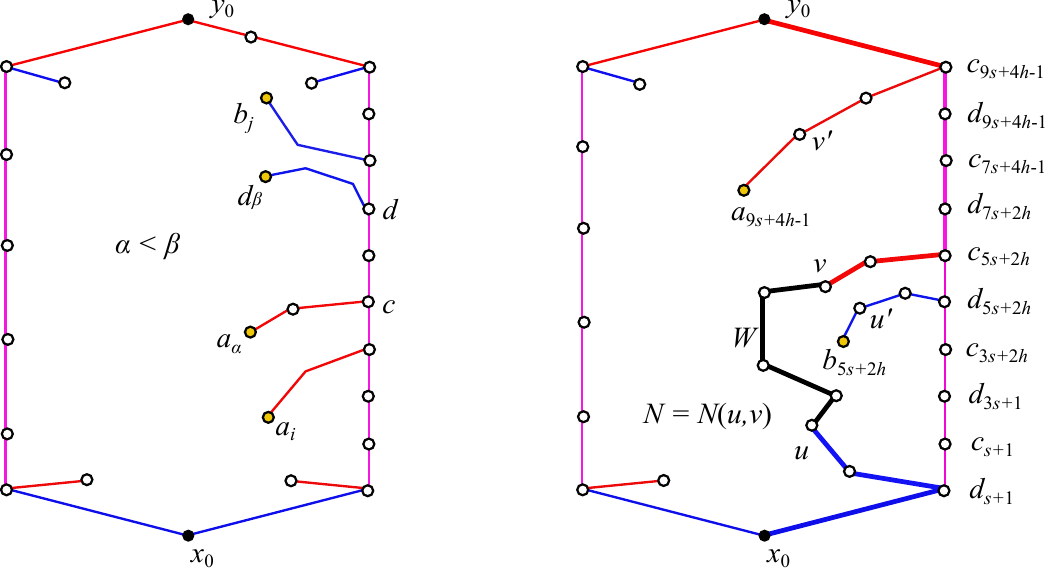}
\end{center}
\caption{
Left: When $\alpha<\beta$, 
the path common to $a_{\alpha}Sy_0$ and $x_0Tb_{\beta}$ is a portion 
of the right side of $\cgF$.
Right: $v'$ is left of $N$, while $u'$ is right of $N$.
}
\label{fig:killer}
\end{figure}

Let $s=6h+1$. 
Note that $n\geq 9s+4h-1$.
If $X$ is any subset of $[n]$ with $|X|=s+1$, then
Proposition~\ref{pro:disjoint-trees} implies that there
are distinct integers $\alpha,\beta\in X$ such that
$x_0Tb_\beta$ intersects $a_\alpha Sy_0$.  We apply this observation to
the set $X=[s+1]$.  We give the balance of the argument under
the assumption that $\alpha<\beta$.  From the details of the argument (illustrated on the right side of Figure~\ref{fig:killer}),
it will be clear that the proof when $\alpha>\beta$ is symmetric.

For all $i\in[n]$, let $c_i$ be the least element of $P$ common to $a_iSy_0$ and the right side of $\cgF$.
Dually, for all $j\in[n]$, let $d_j$ be the greatest element of $P$ common to $x_0Tb_j$ and the right side of $\cgF$.
Since $a_1<_S a_2 <_S \cdots <_S a_n$, we have 
\[
c_1 \leq c_2 \leq \cdots \leq c_n\ \text{ in $P$.}
\] 
Symmetrically, since $b_1<_T b_2 <_T \cdots <_T b_n$, we have 
\[
d_1 \leq d_2 \leq \cdots \leq d_n\ \text{in $P$.}
\]
Note also that since $a_i\parallel b_i$ in $P$, we must have
\begin{equation}
d_i<c_i\ \text{in $P$, for all $i\in[n]$.}\label{eq:universal}
\end{equation}

From our statements above and our assumption that $\alpha<\beta$, 
we conclude $x_0 < c_{\alpha} \leq d_{\beta} < y_0$ in $P$. 
Also, both $c_{\alpha}$ and $d_{\beta}$ are strictly on the right side of $\cgF$.
Again from the statements above we have 
$d_j$ strictly on the right side of $\cgF$ for all $j\geq\beta$. 
Therefore, the only intersection of 
$x_0Tb_j$ with the left side of $\cgF$ is at $x_0$, 
for all $j\geq \beta$.

\smallskip
\noindent
\textbf{Claim 1.}\quad 
If $i\in\set{s+1,\ldots,n-2s}$, then $c_i < d_{i+2s}$ in $P$.

\begin{proof}
Consider $X'=\set{i,\ldots,i+s}$. 
Since $|X'|=s+1$, Proposition~\ref{pro:disjoint-trees} implies 
that there are distinct $\alpha',\beta'\in X'$ such that 
$a_{\alpha'}Sy_0$ and $x_0Tb_{\beta'}$ intersect. 
Since $\beta'\geq s+1\geq \beta$, 
we conclude that $x_0Tb_{\beta'}$ is disjoint from the left side of $\cgF$ (except $x_0$).
Therefore, 
the intersection of $a_{\alpha'}Sy_0$ and $x_0Tb_{\beta'}$
 is on the right side of $\cgF$, 
 so $\alpha'<\beta'$, and 
$c_{\alpha'} \leq d_{\beta'}$ in $P$.
Consider also $X''=\set{i+s,\ldots,i+2s}$. 
Again $|X''|=s+1$ and therefore as before, 
we conclude that there exist $\alpha'',\beta''\in X''$ such that 
$c_{\alpha''} \leq d_{\beta''}$ in $P$. 

Now we have 
\[
c_{i} \leq c_{\alpha'} \leq d_{\beta'} < c_{\beta'} \leq c_{i+s}  \leq c_{\alpha''} \leq d_{\beta''} \leq d_{i+2s}\ \text{in $P$.}
\]
This proves Claim~1.
\end{proof}

Now consider the following chain of nine inequalites.
\begin{align*}
d_{s+1} &< c_{s+1} < d_{3s+1} < c_{3s+2h} < d_{5s+2h}\\
&< c_{5s+2h} < d_{7s+2h} < c_{7s+4h-1} < d_{9s+4h-1}\\
&< c_{9s+4h-1}\ \text{in $P$.}
\end{align*}
The first, third, fifth, seventh, and ninth follow 
by~\eqref{eq:universal}. 
The second, fourth, sixth, and eighth inequalities follow by Claim~1.

Let $X_1=\set{3s+1,\ldots,3s+2h}$ and 
$X_2=\set{7s+2h,\ldots,7s+4h-1}$.
Note that $|X_1|=|X_2|=2h$. 

Let $N=N(u,v)$ be a separating path associated with 
$a_{5s+2h}< b_{s+1}$ in $P$ such that 
$x_0Tu=x_0Nu$ and $vSy_0=vNy_0$. 
Let $W=vNu$. 

Now consider an element $b_j$ with $j\in X_2$.
Recall that $x_0Tb_j$ leaves the right side of $\cgF$ at element $d_j$.
Since $c_{5s+2h} < d_{7s+2h}\leq d_j$ in $P$, 
the first edge of $x_0Tb$ 
that is not on the right side of $\cgF$ 
is left of $N$.
Proposition~\ref{pro:self-evident} implies that 
either all edges and vertices 
of $d_jTb_j$ (except $d_j$) are left of $N$ or 
$d_jTb_j$ intersects $N$ at a point distinct from $d_j$. 
We want to rule out the second possibility. 
Suppose to the contrary that $w\neq d_j$ is an element common to 
$d_jTb_j$ and $N$. 
Then $w$ is in the interior of $\cgF$. 
The only section of $N$ in the interior of $\cgF$ is 
$d_{s+1}Nc_{5s+2h}$, except its endpoints.
Note that $w$ cannot be in $vNc_{5s+2h}$ as 
this would imply $v<c_{5s+2h}<d_j<v$ in $P$. 
Clearly that is impossible.
Next, $w$ is not on $d_{s+1}Nu$ as $T$ is a tree. 
Finally, if
$w$ is on $W=uNv$ then 
$W$ intersects $x_0Tb_j$ and $x_0Tb_{s+1}$, so
Proposition~\ref{pro:no-skip} implies 
$W$ intersects $x_0Tb_{5s+2h}$, which implies 
$a_{5s+2h}\leq u \leq w \leq b_{5s+2h}$ in $P$, 
a clear contradiction.
Thus, indeed $b_j$ is left of $N$, for all $j\in X_2$.

\smallskip
\noindent
\textbf{Claim 2.}\quad 
$a_{9s+4h-1}Sy_0$ does not intersect $W$.

\begin{proof}
Suppose the conclusion does not hold and
take $w$ to be 
the greatest element in $P$ common to $W$ and $a_{9s+4h-1}Sy_0$.
Note that $D=wSc_{9s+4h-1}Nw$ is a cycle in $G$ 
with $D\subseteq U_P(a_{5s+2h})$. 
Proposition~\ref{pro:strict-alternating} implies that 
all elements $a_i$ with $i\neq 5s+2h$ are in the exterior of $D$. 

Consider an element $a_i$ with $i\in X_2$. 
Since $c_{5s+2h}<c_{7s+2h} \leq c_i \leq c_{7s+4h-1}<c_{9s+4h-1}$ in $P$, 
we conclude that the first edge of $c_iSa_i$ is in the interior of $D$. 
Since $a_i$ is in the exterior of $D$, 
the path $c_iSa_i$ has to intersect $D$ in a vertex distinct from $c_i$. 
Let $z$ be the least element common to $a_iSc_i$ and $D$. 
Since all of $D$, except $uWw$, is contained in the tree $S$,  
we must have $z$ in the interior of $uWw$.
It follows that all edges and vertices of $a_iSz$, except $z$ 
are in the exterior of $D$. 

Let $e_0$ and $e_1$ be edges incident to $z$ traversing $D$ 
counterclockwise. 
Let $e$ be the first edge of $zSa_i$. 
Now since $e$ is in the exterior of $D$, 
we have $e$ right of $e_1$ in the $(z,e_0)$-ordering.
Since $e_0$ and $e_1$ are consecutive edges in $N$, 
we conclude that $e$ is right of $N$.

We assert that $a_i$ is right of $N$. 
If this statement fails to hold then there is an element 
$z'$ distinct from $z$ common to $N$ and $zSa_i$.
Note that $z'$ is not in $vNc_{5s+2h}=vSc_{5s+2h}$ as 
$S$ is a tree. 
Note also that $z'$ is not in $d_{s+1}Nu=d_{s+1}Tu$ as 
$\cgF$ is a $Z$-face.
Also $z'$ is not in $vNz=vWz$ by the choice of $z$. 
Finally, $z'$ is not in $zNu=zWu$ as this way
$zWz'$ implies $z< z'$ in $P$, and
$z'Sz$ implies $z' < z$ in $P$, which is a contradiction.
This completes the proof that 
$a_i$ is right of $N$.

Thus, all elements of 
$A(X_2)$ are right of $N$.
However, this is impossible, by Proposition~\ref{pro:separate} as $N$ now separates $A(X_2)$ from $B(X_2)$, and $|X_2|>2h-1$.
This completes the proof of Claim~2.
\end{proof}

Let $N'=N'(u',v')$ be a separating path associated with  
$a_{9s+4h-1}<b_{5s+2h}$ such that 
$x_0Nu'=x_0Tu'$ and $v'Ny_0=v'Sy_0$.
Let $W'=v'Nu'$. 
The proof of the following claim is dual to the proof of Claim~2.

\smallskip
\noindent
\textbf{Claim 3.}\quad 
$x_0Tb_{s+1}$ does not intersect $W'$.

Note that $u'\parallel a_{5s+2h}$ in $P$. 
Indeed, if $a_{5s+2h}\leq u'$ in $P$ then $a_{5s+2h}\leq u'\leq b_{5s+2h}$ in $P$ 
which is false.
Also, if $u'\leq a_{5s+2h}$ in $P$, 
then $x_0 < u' \leq a_{5s+2h} < y_0$ in $P$, 
so we have a chain from $x_0$ to $y_0$ that goes through $u'$ which is in the interior of $\cgF$. This contradicts the fact that $\cgF$ is a $Z$-face.
Now Proposition~\ref{pro:everything} implies that $u'$ is right of $N$.

Now we argue that $v'$ is left of $N$. 
Proposition~\ref{pro:everything} implies that 
either $v'$ is left of $N$ or 
$v'Sy_0$ intersects $x_0Nv$. 
However, Claim~2 forbids intersections of $v'Sy_0$ and $uNv=uWv$. 
Also, $v'Sy_0$ does not intersect $x_0Nu=x_0Tu$ as the first branches of the right side at $c_{9s+4h-1}$ and the latter branches of the right side at $d_{s+1}$.
Thus, $v'$ is left of $N$.

Finally, recall that $W'$ is a witnessing path from $v'$ to $u'$. 
Since $v'$ is left of $N$ and $u'$ is right of $N$, $W'$ must intersect $N$, say at element $w$. Since all elements on $W'$ are interior in $\cgF$, $w$ must be in the section $d_{s+1}Nc_{5s+2h}$. 
However, Claim~3 says that $W'$ is disjoint from $d_{s+1}Nu=d_{s+1}Tu$. 
So $w$ must be in $uNc_{5s+2h}$. 
But now we have 
\[
a_{5s+2h}\leq c_{5s+2h}\leq w \leq u'\leq b_{5s+2h}\ \text{in $P$},
\]
which is the final contradiction.
\end{proof}
And as noted
previously, this completes the proof of Lemma~\ref{lem:36}, as well as 
the principal theorem of the paper.

\section{Closing Comments}

Since we have not been able to disprove that $\dim(P) = \Oh(h)$ 
we comment that our proof for $\Oh(h^6)$
has three steps where improvements might be possible.
Do we really need the $\Oh(h^3)$ factor in the transition from
singly constrained to doubly constrained set of incomparable pairs?  
When $I$ is a set of doubly constrained pairs, 
did we need another factor of $h$ to
transition to the doubly exposed case? 
Could $\dim(I)$ be linear in 
$\rho(\mathbb{F})$ when $\mathbb{F}=(P,G,x_0,y_0,I,\cgD)$ is doubly exposed? 

Although we believe the establishment of a polynomial bound
for dimension in terms of height for posets with planar cover 
graphs is intrinsically interesting,
we find the results of Section~\ref{sec:dim-and-breadth},
where height plays no role, particularly intriguing.
Indeed, we hope that insights from this line of research
may help to resolve the following long-standing conjecture.

The \emph{standard example number} $\se(P)$ of a poset $P$ is 
$1$, if $P$ does not contain a copy of the standard example $S_2$; 
otherwise it is the largest integer $d\geq2$ such that 
$P$ contains a copy of $S_d$.
Clearly, $\se(P)\leq \dim(P)$, for all posets $P$. 
A class $\cgC$ of posets is \emph{dim-bounded} if 
there is a function $f$ such that 
$\dim(P)\leq f(\se(P))$, for all $P$ in $\cgC$.

\begin{conjecture}\label{con:dim-boundedness}
The class of posets with planar cover graphs is dim-bounded.
\end{conjecture}

Apparently, the first reference in print to 
Conjecture~\ref{con:dim-boundedness} is in an informal comment on 
page 119 of~\cite{Tro-book}, published in 1991.  However, the problem 
goes back at least 10 years earlier.  In 1978, Trotter~\cite{Trot78a}
showed that there are posets that have large dimension and have planar
cover graphs.  In 1981, Kelly~\cite{Kel81} showed that there are posets
that have large dimension and have planar order diagrams.  In both of
these constructions, the fact that the posets have large dimension
is evidenced by large standard examples that they contain.  The belief
that large standard examples are necessary for large dimension
among posets with planar cover graphs grew naturally from these
observations.

To attack Conjecture~\ref{con:dim-boundedness}, it
is tempting to believe that we can
achieve a transition from a singly constrained poset to
a doubly exposed poset, independent of height, by allowing
a considerable reduction in the dimension $d$.

\medskip
\noindent
\textbf{Acknowledgment.}
The authors are grateful to Stefan Felsner, Tomasz Krawczyk, Micha{\l} Seweryn, and Heather
Smith Blake for helpful conversations concerning arguments for
the results in this paper.  We are also grateful to two anonymous referees for a careful reading of a preliminary version of this paper and for their helpful suggestions for improvements.

\bibliographystyle{plain}
\bibliography{height-poly-r2}

\end{document}